\documentclass[12pt,twoside,a4paper]{article} 
\usepackage{amsfonts}

\usepackage[bookmarksopen=false,breaklinks=true,%
      backref=page,pagebackref=true,plainpages=false,%
      hyperindex=true,pdfstartview=FitH,colorlinks=true,%
      pdfpagelabels=true,colorlinks=true,linkcolor=blue,%
      citecolor=red,urlcolor=red,hypertexnames=false%
      ]%
   {hyperref}

\title{ Hidden constructions in abstract algebra,\\
Krull Dimension, Going Up, Going Down}
\author{
Thierry Coquand
(\thanks {~
Chalmers, University of G\"oteborg, Sweden,
email: coquand@cs.chalmers.se}~)
Henri Lombardi
(\thanks{~
Equipe de Math\'ematiques, CNRS UMR 6623, UFR des Sciences et
Techniques,
Universit\'e de Franche-Comt\'e, 25 030 BESANCON cedex, FRANCE,
email: lombardi@math.univ-fcomte.fr}~),
}
\date{08 2001}

\DeclareSymbolFont{lasy}{U}{lasy}{m}{n}
\SetSymbolFont{lasy}{bold}{U}{lasy}{b}{n}
\let\Box\undefined
\DeclareMathSymbol\Box{\mathord}{lasy}{"32}

\newtheorem{theorem}{Theorem}[section]

\newtheorem{proposition}[theorem]{Proposition}
\newtheorem{lemma}[theorem]{Lemma}
\newtheorem{corollary}[theorem]{Corollary}

\newtheorem{fact}[theorem]{Fact}
\newtheorem{definition}[theorem]{Definition}

\newtheorem{notation}[theorem]{Notation}

\newcommand {\junk}[1]{}

\newenvironment{proof}[1]{
\trivlist \item[\hskip \labelsep{\bf #1.}]\hskip 0pt\\}{\hfill 
\mbox{$\Box$}
\endtrivlist}
\makeatletter

\def\.@{\char'76}

\def \noi {\noindent}
\def \ss {\smallskip}
\def \sni {\ss\noi}
\def \ms {\medskip}
\def \mni {\ms\noi}
\def \bs {\bigskip}
\def \bni {\bs\noi}

\def \alb {\allowbreak}

\def \aoe {{\land,\lor,\exists}}

\def \these {\quad\vdash\quad}
\def \thesu {\quad\vdash_\aoe\quad}

\def \cl {{\circ}}

\def\equidef{\buildrel{{\rm def}}\over{\quad\Longleftrightarrow\quad}}

\def\gen#1{\left\langle{#1}\right\rangle}

\def \snic#1 {\sni\centerline{$#1$}\ss}
\def \snif#1#2#3 {\vspace{#1}\noindent\centerline{$#3$}\vspace{#2}}

\def \fait#1#2 {\vspace{.1cm} \begin{tabular}{p{2cm}p{2cm}l l l}
$\cl\;#1\;\cl$ & $\these$ & $#2$ \end{tabular}\vspace{.1cm}}
\def \faut#1#2 {\vspace{.1cm} \begin{tabular}{p{2cm}p{2cm}l l l}
               $\cl\;#1\;\cl$ & $\thesu$ & $#2$
                  \end{tabular}}

\def \ov#1 {\overline#1 }

\def \wi#1 {\widetilde#1 }
\def \sd#1#2 {\widetilde#1_#2 }

\def \Z{\mathbb{Z}}
\def \N{\mathbb{N}}

\def \cC {{\cal C}}

\def \cI {{\cal I}}

\def \cN {{\cal N}}
\def \cP {{\cal P}}
\def \cQ {{\cal Q}}

\def \cM {{\cal M}}

\def \cS {{\cal S}}

\def \vu {\,\vee\,} 
\def \vi {\,\wedge\,} 
\def \Vu {\bigvee}
\def \Vi {\bigwedge}
\def \vda {\,\vdash\,}

\def \Un {{\bf 1}}
\def \Deux {{\bf 2}}
\def \Trois {{\bf 3}}
\def \Quatre {{\bf 4}}

\def \Pf {{{\rm P}_{{\rm f}}}}

\def \Hom {{\rm Hom}}

\def \dim {{\rm dim}}

\def \Spec {{\rm Spec}}
\def \Zar {\,{\rm Zar}}
\def \Kr {\,{\rm Kr}}
\def \Kru {\,{\rm Kru}}

\def \num {{n$^{{\rm o}}$}}

\def \hdrz {induction hypothesis}
\def \cad {\textit{i.e.}, }
\def \ssi {if, and only if, }

\def \Propeq {The following properties are equivalent: }
\def \propeq {the following properties are equivalent: }
\def \disept {17$^{{\rm th}}$ Hilbert's problem }

\def \coli {linear combination }

\def \com {comaximal }
\def \comz {comaximal}

\def \entrel {entailment relation }
\def \entrelz {entailment relation}
\def \entrels {entailment relations }
\def \entrelsz  {entailment relations}

\def \homo {homomorphism }

\def \homos {homomorphisms }

\def \idep {prime ideal }

\def \itf {\tf ideal }
\def \itfs {\tf ideals }
\def \itfz {\tf ideal}

\def \mo {monoid }
\def \moco {\com \mos}
\def \mocoz {\com \mosz}
\def \mos {monoids }
\def \mosz {monoids}
\def \moz {monoid}

\def \nst {Null\-stellen\-satz }
\def \nstz {Null\-stellen\-satz}

\def \polcar {characteristic polynomial }

\def \proi {idealistic prime }
\def \prois {idealistic primes }
\def \proiz {idealistic prime}
\def \proisz {idealistic primes}

\def \proc {idealistic chain }
\def \procs {idealistic chains }

\def \procz {idealistic chain}
\def \procsz {idealistic chains}

\def \proel {elementary \proc}
\def \proels {elementary \procs}
\def \proelz {elementary \procz}

\def \proelos {\proels of length }

\def \prolo {\proc of length }
\def \prolos {\procs of length }

\def \tf {finitely generated }

\def \trdi {distributive lattice }
\def \trdis {distributive lattices }
\def \trdiz  {distributive lattice}
\def \trdisz  {distributive lattices}

\def \cov {cons\-truc\-ti\-ve }

\def \LLPO{{\bf LLPO}}

\def \pte {excluded middle principle }
\def \ptez {excluded middle principle}

\def \tcg {compactness theorem }
\def \tcgz {compactness theorem}

\def \Tcgi {The \tcg implies the following result. }

\begin{document}
\maketitle

\begin{abstract}
We present constructive versions of Krull's dimension theory for
commutative rings and distributive lattices. The foundations of these
constructive versions are due to Joyal, Espan\~ol and the authors.
We show  that this gives
a constructive version of basic classical theorems (dimension
of finitely presented algebras, Going up and Going down theorem, 
\ldots),
and hence
that we get an explicit computational content where these abstract 
results are
used to show the existence of concrete elements. This can be seen
as a partial realisation of Hilbert's program for classical abstract
commutative algebra.
\end{abstract}
\bni MSC 2000: 13C15, 03F65, 13A15, 13E05


\bni Key words: Krull dimension, Going Up, Going Down,
Constructive Mathematics.

\newpage

\tableofcontents

\newpage

\section*{Introduction} \label{sec Introduction}
\addcontentsline{toc}{section}{Introduction}
We present constructive versions of Krull's dimension theory for
commutative rings and distributive lattices. The foundations of these
constructive versions are due to Joyal, Espan\~ol and the authors.
We show  that this gives
a constructive version of basic classical theorems (dimension
of finitely presented algebras, Going up and Going down theorem, 
\ldots),
and hence
that we get an explicit computational content when these abstract 
results are
used to show the existence of concrete elements. This can be seen
as a partial realisation of Hilbert's program for classical abstract
commutative algebra.

 Our presentation follows Bishop's style (cf. in algebra \cite{MRR}).
As much as possible, we kept minimum any explicit mention to logical 
notions.
When we say that we have a constructive version of an abstract
algebraic theorem, this means that we have a theorem the proof of
which is constructive, which has a clear computational content, and
from which we can recover the usual version of the abstract theorem
by an immediate application of a well classified non constructive
principle. An abstract
classical theorem can have several distinct interesting constructive
versions.

 In the case of abstract theorem in commutative algebra, such a non
constructive principle is the compactness theorem, which
claims the existence of a model of a formally consistent propositional
theory\footnote{Mathematically, this result can be seen as stating
the compactness of product spaces $\{0,1\}^V$; thus it can be seen
as a special case of Tychonov's theorem.}.
When this is used for algebraic structures of enumerable
presentation (in a suitable sense) this principle is
nothing else than a reformulation of Bishop $\LLPO$ (a
real number is $\geq 0$ or $\leq 0$).

 To avoid the use of \tcg is not
motivated by philosophical but by practical
considerations. The use of this principle leads indeed to replace
quite direct (but usually hidden) arguments by indirect ones which are
nothing else than a double contraposition of the direct proofs,
with a corresponding lack of computational content.
For instance \cite{clr} the abstract proof of \disept claims~: if
the polynomial $P$ is not a sum of rational fractions there is
a field $K$ in which one can find an absurdity by reading the 
(constructive)
proof that the polynomial is everywhere positive or zero. The direct
version of this abstract proof is: from the (constructive) proof
that  the polynomial is everywhere positive or zero, one can show
(using arguments of the abstract proofs) that any attempt to build $K$
will fail. This gives explicitly the sum of squares we are looking for.
In the meantime, one has to replace the abstract result: ``any real
field can be ordered'' by the constructive theorem: ``in a field
in which any attempt to build an ordering fails $-1$ is a sum of 
squares''.
One can go from this explicit version to the abstract one by \tcgz, 
while
the proof of the explicit version is hidden in the
algebraic manipulations that appear in the usual classical proof of the
abstract version.

\mni Here is the content of the paper.

\paragraph{Pseudo regular sequences and Krull dimension of commutative 
rings}~


\noindent In  section \ref{secKrA} we give ``more readable'' proofs for
some of the results contained in \cite{lom}, which were there proved
using the notion of dynamical structures \cite{clr}. The central notion 
which
is used is the one of {\em partial specification of a chain of ideal
primes}. Abstract constructions on chains of prime ideals are then
expressed constructively in the form of simultaneous collapsing theorems
\cite{clr}
(theorem \ref{ThColSimKrA}).
We present the notion of pseudo-regular sequence (a weakened form
of the notion of regular sequence), which allows us to define 
constructively
the Krull dimension of a ring. We show in this way that the notion of 
Krull
dimension
has an explicit computational content in the form of existence (or
non existence) of some algebraic identities. This confirms the feeling 
that
commutative algebra can be seen computationally as a machinery producing
algebraic identities (the most famous of which being called \nstz).
Finally we give a constructive version of the theorem which says that
the Krull dimension of a ring is the upper bound of the Krull
dimension of its localizations along maximal ideals.

\paragraph{Distributive lattices}~

\noindent In section \ref{secKrullTreil}  we develop the theory of
Krull dimension of distributive lattices, first in the style of
 section  \ref{secKrA}, and then in the style of the theory of
Joyal. We then show the connections between these two presentations.
An important simplification of proofs and computations
is obtained via the systematic use of the notion of entailment
relation, which has its origin in the cut rule  in
Gentzen's sequent calculus, with the fundamental theorem 
\ref{thEntRel1}.

\paragraph{Zariski and Krull lattice}

\noindent In  section \ref{secZariKrull}  we define the
Zariski lattice of a commutative ring (whose elements are
radicals of finitely generated ideals), which is the constructive
counterpart of Zariski spectrum~: the points of Zariski spectrum
are the prime ideals of Zariski lattice, and the constructible subsets 
of
Zariski spectrum are the elements of the Boolean algebra generated
by the Zariski lattice.
Joyal's idea is to define Krull dimension of a commutative ring as
the dimension of its Zariski lattice. This avoids any mention of
prime ideals. We show the equivalence between this (constructive) point
of view of the (constructive) presentation given in section
\ref{secKrA}.

\paragraph{Going Up and Going Down}~

\noindent  Section \ref{secGUGD} presents the famous
Going up theorem for integral extensions, and the
Going down  theorem for integral extensions of integrally closed
domains
and for flat extensions.
We show that these theorems, which seem at first quite abstract (since
they claim the existence of some prime ideals) have quite concrete
meaning as constructions of algebraic identities.
These constructive versions may seem 
at first a little strange, but they are
directly suggested by this process of 
making explicit the abstract arguments
of these classical results.

\paragraph{Conclusion}~

\noindent
 This article confirms the actual realisation of Hilbert's program
for a large part of abstract commutative algebra.
(cf.  \cite{clr,cp,kl,lom95,lom97,lom98,lom,lom99,lom99a,lq99}).
The general idea is to replace ideal abstract structures
by {\em partial specifications} of these structures.
The very short elegant abstract proof which uses these ideal
objects has then a corresponding computational version at the
level of the partial specifications of these objects.
Most of classical results in abstract commutative algebras,
the proof of which seem to require in an essential way excluded
middle and Zorn's lemma, seem to have in this way a corresponding
constructive version. Most importantly, the abstract proof of the
classical theorem always
contains, more or less implicitly, the constructive proof of
the corresponding constructive version.

Finally one should note that the constructive theorems which concern
the Krull dimension
of polynomial rings and of finitely presented algebra over a field, the
Going up and Going down are new (they could not be obtained in
the framework of Joyal's theory as long as one only looks at Zariski
lattice without explicitating the computations and algebraic identities
involved there).

\section{Constructive definition of Krull dimension of commutative 
rings}
\label{secKrA}
Let $R$ be a commutative ring. We write
$\gen{J}$  or explicitly
$\gen{J}_R$  the ideal of $R$ generated by the subset $J\subseteq R$. We 
write
$\cM(U)$ the monoid ({\footnote{~A monoid will always be 
multiplicative.}})
generated by the subset $U\subseteq R$.

\subsection{Idealistic chains}
\label{subsecProc}

\begin{definition}
\label{defproch1} In a commutative ring $R$
\begin{itemize}
\item
A {\em partial specification for a prime ideal} (in abreviated form
{\em idealistic prime}) is a couple $\cP= (J,U)$ of subsets of $R$.
\item An \proi $\cP= (J,U)$ is said to be  {\em complete} if $J$ is an 
ideal
$U$ is a  \mo and   $J+U=  U$.
\item Let $\cP_1= (J_1,U_1)$ and $\cP_2= (J_2,U_2)$  be two \proisz.
We say that
{\em $\cP_1$ is contained into $\cP_2$} written
$\cP_1\subseteq \cP_2$ if $J_1\subseteq J_2$ and  $U_2\subseteq U_1$.\\
We say that {\em $\cP_2$ refines $\cP_1$} and we write
$\cP_1\leq \cP_2$ if $J_1\subseteq J_2$ and  $U_1\subseteq U_2$.
\item
A {\em partial specification of a chain of prime ideals}
(in abreviated form {\em \procz}) is defined as follows.
An {\em \prolo $\ell$}  is
a list of $\ell+1$ \proisz:  $\cC= (\cP_0,\ldots,\cP_\ell)$
($\cP_i= (J_i,U_i)$). We shall write $\cC^{(i)}$ for $\cP_i$.
The \proc will be said
{\em finite} iff all sets  $J_i$ and $U_i$ are finite.
\item
An \proc $\cC= (\cP_0,\ldots,\cP_\ell)$ is said to be
{\em complete} iff the
\prois $\cP_i$ are complete and if we have
$\cP_i\subseteq \cP_{i+1}$ $(i= 0,\ldots ,\ell-1)$.
\item
Let $\cC= (\cP_0,\ldots,\cP_\ell)$ be two \prolos $\ell$ and
$\cC'= (\cP'_0,\ldots,\cP'_\ell)$. We say that {\em $\cC'$ is a 
refinement of
$\cC$} and we write
$\cC\leq \cC'$ if $\cP_i\leq \cP'_i$ for $i= 0,\ldots ,\ell$.
\end{itemize}
\end{definition}
We can think of an
 \proc $\cC$ of length $\ell$  in $A$ as a partial specification of
an increasing chain of prime ideals (in the usual sense)
 $P_0,\ldots,P_\ell$ such  that
$\cC\leq (\cQ_0,\ldots \cQ_\ell)$, where $\cQ_i= (P_i,A\setminus P_i)$
$(i= 0,\ldots,\ell)$.
\begin{fact}
\label{factProcComp}
Any \proc
$\cC= ((J_0,U_0),\ldots,(J_\ell,U_\ell))$ generates a complete 
minimal \proc
$$\cC'= ((I_0,V_0),\ldots,(I_\ell,V_\ell)) $$  defined by
$ I_0= \gen{J_0}$, $ I_1= \gen{J_0\cup J_1}$,\ldots,
$I_\ell= \gen{J_0\cup\cdots\cup J_\ell}$, $U'_i= \cM(U_i)$
$(i= 0,\ldots,\ell)$,  $V_\ell= U'_\ell+I_\ell$,
$V_{\ell-1}= U'_{\ell-1}V_\ell+I_{\ell-1}$, \ldots,
$V_0= U'_0V_1+I_0= U'_0(U'_1(\cdots(U'_\ell+I_\ell)+\cdots)+I_1)+I_0$.
Furthermore any element of $V_0$ can be rewritten as
$$u_0\cdot(u_1\cdot(\cdots(u_\ell+j_\ell)+\cdots)+j_1)+j_0= u_0\cdots 
u_\ell+
u_0\cdots u_{\ell-1}\cdot j_\ell+\cdots+u_0\cdot j_1+j_0
$$
with $j_i\in \gen{J_i}$ and $u_i\in\cM(U_i)$.
\end{fact}

\begin{definition}
\label{defConjug}
An ideal $I$ and a
\mo $S$ are said to be {\em conjugate} if we have:
$$\begin{array}{rcl}
(s\cdot a\in I, \; s\in S)   &\Longrightarrow & a\in I   \\
 a^n\in I              &\Longrightarrow & a\in I  \quad (n\in\N,\;n>0) 
\\
(j\in I, \; s\in S)    &\Longrightarrow & s+j\in S   \\
 s_1\cdot s_2\in S           &\Longrightarrow & s_1 \in S
\end{array}$$
In this case we say also that the \proi $(I,S)$ is {\em saturated}.
\end{definition}
For instance a detachable prime ideal
({\footnote{~A subset of a set is said to be detachable if we can decide
membership to this subset. For example the \itfs of a polynomial ring
with integer coefficients are detachable.}})
and the complementary \mo are conjugate.
When an ideal $I$ and a \mo $S$ are conjugate we have
$$ 1\in I\quad \Longleftrightarrow\quad  0\in S\quad
\Longleftrightarrow\quad  (I,S)= (A,A)
$$
\begin{definition}
\label{defproch2}  {\em (Collapsus)}
Let  $\cC= ((J_0,U_0),\ldots,(J_\ell,U_\ell))$ be an \proc and
$\cC'= ((I_0,V_0),\alb\ldots,\alb(I_\ell,V_\ell))$ the complete \proc
it generates.
\begin{itemize}
\item  We say that the \proc $\cC$ {\em collapses}
iff we have
$0\in V_0$. An alternative definition is: there exists
$j_i\in \gen{J_i}$, $u_i\in\cM(U_i)$, $(i= 0,\ldots,\ell)$
satisfying the equation
$$u_0\cdot(u_1\cdot(\cdots(u_\ell+j_\ell)+\cdots)+j_1)+j_0= 0
$$
Such a relation is called a {\em collapsus} of the \proc $\cC$.
\item An \proc is said to be {\em saturated} iff it is complete
and the \prois $(J_i,U_i)$ are saturated.
\item  The  \proc
$((A,A),\ldots,(A,A))$ is said to be {\em trivial}:
a saturated chain that collapses is trivial.
\end{itemize}
\end{definition}

Notice that the \proi $(0,1)$ collapses if and only if $1= _A0$.

The following lemma is direct.
\begin{fact}
\label{fact}
\label{factColAc}
An \proc $\cC= ((J_0,U_0),\ldots,(J_\ell,U_\ell))$ with
$U_h\cap J_h\neq \emptyset$  for some $h$ collapses. More generally
if an \proc $\cC'$ extracted from an \proc $\cC$ collapses then
$\cC$ collapses. Similarly if
$\cC$ collapses, any refinement of $\cC$ collapses.
\end{fact}
The proofs of the following properties are instructive.
\begin{fact}
\label{factSatur}
Let $\cC_1= (\cP_0,\ldots,\cP_\ell)$ and
$\cC_2= (\cP_{\ell+1},\ldots,\cP_{\ell+r})$
be two  \procs on a ring $A$. Let $\cC$ be $\cC_1\bullet
\cC_2= (\cP_0,\ldots,\cP_{\ell+r})$.
\begin{itemize}
\item [$(1)$] Suppose that $\cC_1$ is saturated. Then $\cC$ collapses in 
$A$
\ssi $\cP_\ell\bullet \cC_2$   collapses in  $A$
\ssi $\cC_2$   collapses in the quotient $A/I_\ell$
($\cP_\ell= (I_\ell,U_\ell)$).
\item [$(2)$] Suppose that $\cC_2$ is complete. Then $\cC$ collapses in 
$A$
\ssi $ \cC_1\bullet\cP_{\ell+1}$   collapses in $A$
\ssi $\cC_2$   collapses in the localisation $A_{U_{\ell+1}}$
($\cP_{\ell+1}= (I_{\ell+1},U_{\ell+1})$).
\item [$(3)$] Suppose that $\cC_1$ is saturated and $\cC_2$ is complete.
Then $\cC$ collapses in $A$  \ssi
$(\cP_\ell,\cP_{\ell+1})$   collapses in $A$
\ssi $I_\ell\cap U_{\ell+1}\not=  \emptyset$.
\end{itemize}
\end{fact}
\begin{proof}{Proof}
Left to the reader.
\end{proof}
\subsection{Simultaneous collapse}
\label{subsecColsim}
\begin{notation}
\label{notaAjou}
{\rm  In the sequel we will use the following notations
for an \proi or an \proc obtained by refinement.
If $\cP= (J,U)$, $\cC= (\cP_1,\ldots \cP_n)$, we write
\begin{itemize}
\item  $(J,x;U)$ or still $\cP\,\& \left\{x\in \cP \right\}$
for $(J\cup \left\{x\right\}, U)$
\item  $(J;x,U)$  or still $\cP\,\& \left\{x\notin \cP \right\} $
for $(J,U\cup \left\{x\right\})$
\item   $\cP\,\& \left\{I\subseteq \cP \right\} $
for $(J\cup I, U)$
\item   $\cP\,\& \left\{V\subseteq A\setminus\cP \right\} $
for $(J, U\cup V)$
\item  $\cC\,\& \left\{x\in \cC^{(i)} \right\} $
for $(\cP_0,\ldots,\cP_i\,\& \left\{x\in \cP_i \right\},\ldots,\cP_n)$
\item   $\cC\,\& \left\{x\notin \cC^{(i)} \right\} $
for $(\cP_0,\ldots,\cP_i\,\& \left\{x\notin \cP_i 
\right\},\ldots,\cP_n)$
\item etc\dots
\end{itemize}
}
\end{notation}

\begin{theorem}
\label{ThColSim1}{\em (Simulateneous collapse for an \proiz)}\\
Let $\cP= (J,U)$ be an \proi in a commutative ring $R$.
\begin{itemize}
\item [$(1)$] Let $x$ be an element of $R$. Suppose that the \prois
$\cP\,\& \left\{x\in \cP \right\} $
and $\cP\,\& \left\{x\notin \cP \right\} $
both collapse, then so does $\cP$.
\item [$(2)$] The \proi $\cP$ generates a minimum saturated \proiz.
We get it by adding in
$U$ (resp. $J$) any element $x\in R$  such that the \proi
$\cP\,\& \left\{x\in \cP \right\}$
(resp. $\cP\,\& \left\{x\notin \cP \right\}$) collapses.
\end{itemize}
\end{theorem}
\begin{proof}{Proof}
The proof of point (1) is Rabinovitch trick. From the two equalities
$u_1+j_1+ax= 0$ and $u_2x^m+j_2= 0$ (with $u_i\in \cM(U)$, $j_i\in 
\gen{J}$,
$a\in R$,
$n\in\N$) we build a third one, $u_3+j_3= 0$ by eliminating $x$:
we get $u_2(u_1+j_1)^m+(-a)^mj_2= 0$, with $u_3= u_2u_1^m$. \\
The point (2) is seen to be a consequence of (1) as follows.
Let
$\cP'= (\gen{J},\gen{J}+\cM(U))$ be the complete \proi generated by 
$\cP$.
Let $\cP''= (I',S')$ be a saturated \proi which refines $\cP$. Let
$\cP_1= (K,S)$ the \proi described in (2).
It is easily seen that we have   $\cP\leq \cP'\leq \cP_1\leq \cP''$. 
Thus
we are left
to check that $\cP_1$ is a saturated \proiz. We shall see that this 
results
from (1)
without having to do any computations.
Let us show for instance that
$K+K\subseteq K$. Let $x$ and $y$ be in  $K$ \cad
such that $(I;x,U)$ and
$(I;y,U)$ both collapse. We have to show that $(I;x+y,U)$ also 
collapses.
For this, by (1), it is enough to show that
$\cP_2= (I,x;x+y,U)$ and $\cP_3= (I;x+y,x,U)$ collapse. For $\cP_3$ it 
is
by hypothesis. If we complete $\cP_2$, we get $y= x+y-x$
in the \moz, hence this is a refinement of $(I;y,U)$ which collapses by
hypothesis.\\
The other verifications are direct, by similar arguments.
\end{proof}

Notice that the saturation of
$(0,1)$ is $(\cN,A^{\times })$ where $\cN$ is the
nilradical of $R$ and $A^{\times }$
the group of units.
\begin{corollary}
\label{corNstformHilbert} {\em (Krull's theorem or formal Hilbert \nst)} 
\\
Let $\cP= (J,U)$ be an \proi in a commutative ring $R$. Compactness 
theorem
implies that \propeq
\begin{itemize}
\item For all $j\in \gen{J}$, $u\in\cM(U)$, we have $u\neq j$.
\item There exists a detachable prime ideal $Q$ such that
$\cP\le  Q$, \cad such that $J\subseteq Q$ and $U\cap Q= \emptyset$.
\item There exists an \homo $\psi$ from $R$ to an entire ring $S$
such that $\psi(J)= 0$ and $0\notin\psi(U)$.
\end{itemize}
Furthermore if the saturation of $\cP$ is $(I,V)$, \tcg
implies that $I$ is the intersection of all prime 
ideals containing $J$  and disjoint from $U$, whereas
 $V$ is the union of all complement of these prime ideals.
\end{corollary}
\begin{proof}{Proof}
See the proof of theorem \ref{th.nstformel} (which is more general).
\end{proof}
\begin{theorem}
\label{ThColSimKrA}{\em (Simultaneous collapse for the \procsz)}\\
Let $\cC= ((J_0,U_0),\ldots,(J_\ell,U_\ell))$ be an \proc in a 
commutative
ring $R$
\begin{itemize}
\item [$(1)$] Suppose $x\in R$ and $i\in\left\{0,\ldots ,\ell\right\}$.
Suppose that the \procs
$\cC\,\& \left\{x\in \cC^{(i)} \right\} $
and
$\cC\,\& \left\{x\notin \cC^{(i)} \right\} $
both collapse, then so does $\cC$.
\item [$(2)$] The \proc $\cC$ generates a minimum saturated \procz.
We get it by adding in
$U_i$ (resp. $J_i$) all element $x\in R$  such that the \proc
$\cC\,\& \left\{x\in \cC^{(i)} \right\} $
(resp. $\cC\,\& \left\{x\notin \cC^{(i)} \right\} $)
collapses.
\end{itemize}
\end{theorem}
\begin{proof}{Proof}
Let us write $(1)_\ell$ and $(2)_\ell$ the statements for a fixed $\ell$.
Notice that $(1)_0$ and $(2)_0$ are proved in theorem
\ref{ThColSim1}. We shall reason then by induction on $\ell$.
We can suppose that the \proc $\cC$ is complete (since a chain collapses
iff its completion does).\\
The fact that $(1)_\ell\Rightarrow (2)_\ell$ is direct, and can be 
proved by
similar arguments as in the proof of theorem \ref{ThColSim1}. \\
We are left to show
$((1)_{\ell-1}\; {\rm and}\; (2)_{\ell-1})\,\Rightarrow\, (1)_{\ell}$ 
(for
$\ell>0$).\\
If $i<\ell$
we  use  fact \ref{factSatur}: we have then \prolos $i$
in the localisation $A_{U_{i+1}}$ and we can apply the \hdrz.\\
If $i= \ell$ we consider the \prolo $\ell-1$
$((K_0,S_0),\alb\ldots,\alb(K_{\ell-1},S_{\ell-1}))$ that we get by 
saturating
$((J_0,U_0),\ldots,(J_{\ell-1},U_{\ell-1}))$ (we use $(2)_{\ell-1}$).
For arbitrary $j_i\in \gen{J_i}$ and
$u_i\in\cM(U_i)$  ($0\le i\le\ell$),
let us consider the following assertions:
$$ u_0\cdot(u_1\cdot(\cdots(u_{\ell-1}\cdot (u_\ell+j_\ell)+ j_{\ell-1})
+\cdots)+j_1)+j_0= 0\qquad (\alpha)
$$
$$ (u_\ell+j_\ell)\in K_{\ell-1}\qquad (\beta)
$$
$$ \exists n\in\N\;\; u_0\cdot(u_1\cdot(\cdots(u_{\ell-1}\cdot
(u_\ell+j_\ell)^n+ j_{\ell-1}) +\cdots)+j_1)+j_0= 0\qquad (\gamma)
$$
We have $(\alpha) \Rightarrow (\beta) \Rightarrow (\gamma)$.
Hence the following properties are equivalent. Primo: the \proc
$\cC$ collapses in $R$ (this is certified by an equality of type 
$(\alpha)$
or $(\gamma)$). Secundo: the \proi $(J_\ell,U_\ell)$ collapses in
$R/K_{\ell-1}$ (which is certified by an equality of type $(\beta)$). We
are then
reduced to the case $(1)_0$ in the ring  $R/K_{\ell-1}$, and this case 
has
been
dealt with in theorem \ref{ThColSim1}.
\end{proof}
Notice that we have made no use in the end of this argument of
fact \ref{factSatur}, which is not powerful enough in this situation.
The following facts are simple corollaries of theorem
\ref{ThColSimKrA}. Notice that the second point allows an improved
utilisation of fact  \ref{factSatur}.
\begin{fact}
\label{corColSimKrA}~
\begin{itemize}
\item  An \proc $\cC$  collapses \ssi any saturated \proc which refines
$\cC$ is trivial.
\item One does not change the collapsus of an \proc $\cC$ if we replace 
a
subchain by its saturation.
\item  Suppose $x_1,\ldots,x_k\in R$ and that the \procs
$((J_0,U_0),\ldots,(J_i\cup\{(x_h)_{h\in H}\},
U_i\cup\{(x_h)_{h\in H'}\}), \ldots,(J_\ell,U_\ell))$
collapse for all complement pair $(H,H')$ of $\{1,\ldots,k\}$,
then $\cC$ collapses.
\end{itemize}
\end{fact}
\begin{definition}
\label{defproch3} ~
\begin{itemize}
\item  Two \procs that generate the same saturated \proc are said to be 
{\em
equivalent}.
\item An \proc of {\em finite type} is an \proc
  which is equivalent to a finite \procz.
\item An \proc is {\em strict} iff
$V_i\cap I_{i+1}\neq \emptyset$  $(i= 0,\ldots,\ell-1)$ in the 
corresponding
generated saturated \procz.
\item  A saturated \proc $\cC= ((J_0,U_0),\ldots,(J_\ell,U_\ell))$ is
  {\em frozen} if it does not collapse and if for all
$i= 0,\ldots,\ell$, $J_i\cup U_i= R$. An \proc is frozen iff its 
saturation is
frozen. To give a frozen \proc is equivalent to give an increasing chain
of detachable prime ideals.
\end{itemize}
\end{definition}

 We think that theorem
\ref{ThColSimKrA} reveals a computational content which is ``hidden''
in usual classical proofs about increasing chains of prime ideals.
We illustrate this point in the following theorem, which, in
classical terms, gives a concrete characterisation of \procs which
are incomplete specifications of increasing chains of prime ideals.

\begin{theorem}
\label{th.nstformel} {\em (formal \nst for chain of prime ideals)}
Let $R$ be a ring and
$((J_0,U_0),\ldots,(J_\ell,U_\ell))$ an \proc in $R$. The \tcg implies
that the following are equivalent:
\begin{itemize}
\item [$(a)$] There exist $\ell+1$ detachable prime ideals
$P_0\subseteq \cdots\subseteq P_\ell$  such that
 $J_i\subseteq P_i$, $U_i\cap P_i= \emptyset $, $(i= 0,\ldots,\ell)$.
\item [$(b)$] For all $j_i\in \gen{J_i }$ and $u_i\in\cM(U_i)$,
$(i= 0,\ldots,\ell)$
$$u_0\cdot(u_1\cdot(\cdots(u_\ell+j_\ell)+\cdots)+j_1)+j_0\neq 0.
$$
\end{itemize}
\end{theorem}

\begin{proof}{Proof}
Only $(b)\Rightarrow (a)$ is not direct.
Let us start by a proof that uses not the \tcg but \pte and Zorn's 
lemma.
We consider an \proc
$\cC_1= ((P_0,S_0),\ldots,(P_\ell,S_\ell))$
maximal (w.r.t. the refinement relation) among all \proc which refine
$\cC$ and that are not collapsing. It is first clear that $\cC_1$ is
complete, since collapsus is not changed by completion.
If this was not a chain of prime ideals with complement, we would
have for some index $i\;:$
$S_i\cup P_i\neq R$.
In this case let $x\in A\setminus (S_i\cup P_i)$.
Then  $((P_0,S_0),\ldots,(P_i\cup\{x\},S_i),\ldots,(P_\ell,S_\ell))$
has to collapse (by maximality).
The same holds for
$((P_0,S_0),\ldots,(P_i,S_i\cup\{x\}),\ldots,(P_\ell,S_\ell))$.
By theorem \ref{ThColSimKrA} the \proc
$((P_0,S_0),\ldots,(P_\ell,S_\ell))$ collapses, which is absurd.\\
Let us present next a proof which uses only the \tcgz, hence with a
restricted use
of \ptez.\\
We consider the syntactical  propositional theory which describes
an increasing chain of prime ideals of length $\ell$ in $R$.
In this theory we have atomic proposition for $x\in P_i$
and the axioms state that each $P_i$ defines a  proper prime ideal,
that is
\begin{itemize}
\item $\neg (1\in P_i)$
\item $0\in P_i$
\item $a\in P_i\wedge b\in P_i\rightarrow (a+b)\in P_i$
\item $a\in P_i\rightarrow ab\in P_i$
\item $ab\in P_i\rightarrow (a\in P_i\vee b\in P_i)$
\end{itemize}
and we have furthermore $a\in P_i\rightarrow a\in P_{i+1}$ and
$a\in P_i$ for $a\in J_i$ and $\neg (b\in P_i)$ for $b\in U_i$.
By theorem \ref{ThColSimKrA} this theory is consistent.
By the \tcgz, it has a model. This model gives us $\ell+1$
prime ideals as desired.
\end{proof}

 Notice also that theorem \ref{th.nstformel} implies (in two lines)
the simultaneous collapsus theorem \ref{ThColSimKrA}. This last
result can thus be considered to be the \cov version of the first.
A possible corollary of theorem \ref{th.nstformel} would be a
characterisation of the saturated \proc generated by an \proc $\cC$
via the family of chains of prime ideals that are refinements of $\cC$
like in the last claim of theorem \ref{ThColSim1}.
\subsection{Pseudo regular sequences and Krull dimension}
\label{subsecPsr}
\label{nbpIneq}In a constructive framework, it is sometimes better to
consider an inequality relation $x\neq 0$ which is not simply the 
negation
of $x= 0$. For instance a real number is said to be $\neq 0$ iff it is
invertible, \cad apart from $0$. Whenever we mention an inequality 
relation
$x\neq 0$ in a ring, we always suppose implicitly that this relation
has been defined first in the ring. We require that this relation is
a standard apartness relation, that is we ask that, modulo the use
of the excluded middle, it can be shown to be equivalent to 
$\lnot(x= 0)$.
We ask also the conditions
$\; (x\neq 0,\; y= 0)\; \Rightarrow\;  x+y\neq 0$,
$\; xy\neq 0\Rightarrow \; x\neq 0,$ and
$\lnot(0\neq 0)$. Finally  $x\neq y$ is defined as  $x-y\neq 0$.
Without any further precisions on $x\neq 0$ one can always consider
that it is $\lnot(x= 0)$. When the ring is a discrete set, that is when
there is an equality test, we always chose the inequality 
$\lnot(x= 0)$.
Nevertheless it would be a misguided conception to believe that
algebra should limit itself to discrete sets.
\begin{definition}
\label{defproch}
Let  $(x_1,\ldots,x_\ell)$ be a sequence of length $\ell$ in a 
commutative
ring $R$.
\begin{itemize}
\item  The \proc $((0,x_1),\alb(x_1,x_2),\alb\ldots,\alb
(x_{\ell-1},x_\ell),(x_\ell,1))$ is said to be an {\em \proelz}.
It is said to be {\em associate to the sequence $(x_1,\ldots,x_\ell)$}.
We write it $\overline{(x_1,\ldots,x_\ell)}$.
\item
The sequence
$(x_1,\ldots,x_\ell)$ is said to be {\em pseudo singular} when the 
associate
\proel $\overline{(x_1,\ldots,x_\ell)}$ collapses.
This means that there exist
$a_1,\ldots,a_\ell\in R$ and $m_1,\ldots,m_\ell\in \N$ such that
$$ x_1^{m_1}(x_2^{m_2}\cdots(x_\ell^{m_\ell} (1+a_\ell x_\ell) +
\cdots+a_2x_2) + a_1x_1) =   0
$$
\item   The sequence
$(x_1,\ldots,x_\ell)$ is {\em pseudo regular} when the corresponding
\proel does not collapse. This means that for all
$a_1,\ldots,a_\ell\in R$ and all $m_1,\ldots,m_\ell\in \N$ we have
$$ x_1^{m_1}(x_2^{m_2}\cdots(x_\ell^{m_\ell} (1+a_\ell x_\ell) +
\cdots+a_2x_2) + a_1x_1) \neq  0
$$
\end{itemize}
\end{definition}
Notice that the length of the \proel associate to a sequence is the same 
as
the length of this sequence.

The connection with the usual notion  of regular sequence is given by
the following proposition, which is direct.

\begin{proposition}
\label{prop.regseq} In a commutative ring $R$ any regular sequence is
pseudo regular.
\end{proposition}

The following lemma is sometimes usefull.

\begin{lemma}
\label{lem.pseudoreg}
Let  $(x_1,\ldots,x_\ell)$ and $(y_1,\ldots,y_\ell)$ be two sequences in
a commutative ring $R$.
Suppose that for each $j$, $x_j$ divides a power of $y_j$
and that $y_j$ divides a power of $x_j$.
The sequence  $(x_1,\ldots,x_\ell)$ is then pseudo singular \ssi
the sequence  $(y_1,\ldots,y_\ell)$ is pseudo singular.
\end{lemma}
\begin{proof}{Proof}
Indeed, if $x$ divides a power of $y$ and $y$ divides a power of $x$
we have the following refinement relations
$$(a;x)(x;b)  \leq  (a;x,y)(x,y;b) \leq
         {\rm \;the\; saturation\; of\; the\; sequence\; }   
(a;x)(x;b)$$
one add the first $y$ by the relation $yc= x^k$
  ($y$ is hence in the saturation of the \mo generated by $x$),
one add the second by the relation     $y^m= dx$
  ($y$  is hence in the radical of the ideal generated by $x$).
We deduce by symmetry that  $(a;x)(x;b)$ and $(a;y)(y;b)$ have the same
saturation.
\end{proof}

An immediate corollary of theorem \ref{th.nstformel} is the following
theorem  \ref{th.pseudoreg}.

\begin{theorem}
\label{th.pseudoreg}
{\em (pseudo regular sequences and increasing chain of prime ideals)} 
\Tcgi
In a ring $R$ a sequence $(x_1,\ldots,x_\ell)$ is pseudo regular
\ssi there exist  $\ell+1$ prime ideals
$P_0\subseteq \cdots\subseteq P_\ell$
with $x_1\in P_1\setminus P_0$, $x_2\in P_2\setminus P_1, \ldots$
$x_\ell\in P_\ell\setminus P_{\ell-1}$.
\end{theorem}

This leads to the following definition, which gives an explicit
constructive content to the notion of {\em Krull dimension of a ring.}
\begin{definition}
\label{def.dimKrull}~{\em (Krull dimension of a ring)}
\begin{itemize}
\item  A ring $R$ is {\em of dimension $-1$} if, and only if $1= _A0$.  
It is
of dimension $\geq 0$ \ssi  $1\neq_A0$, $>-1$ \ssi  $\lnot(1= _A0)$ and 
$<0$
\ssi  $\lnot(1\neq_A0)$.
\end{itemize}
Let us now suppose $\ell\geq 1$.
\begin{itemize}
\item  A ring is {\em of dimension $\leq \ell-1$} \ssi all \proelos
$\ell$ collapse.
\item  A ring is {\em of dimension $\geq \ell$} \ssi there exists a
pseudo regular sequence de longueur $\ell$.
\item  A ring is {\em of dimension $\ell$} \ssi it is both of
dimension $\ge\ell$ and $\le\ell$.
\item  A ring is {\em of dimension $<\ell$ } \ssi it is impossible for
it to be of dimension $\geq \ell$.
\item  A ring is {\em of dimension $>\ell$ } \ssi it is impossible for
it to be of dimension $\leq \ell$.
{\rm (}{\footnote{~Actually, there exists one and only one \proel of 
length
$0$: $(0,1)$, hence there was no need to begin with a particular 
definition
of ring of dimension $-1$.
In this framework, we recover the distinction between being of dimension
$\ge 0$ and
of dimension $>-1$, as well as the distinction between being of 
dimension
$\le -1$
and dimension $<0$.}}{\rm)}.
\end{itemize}
\end{definition}


A ring is thus of (Krull) dimension $\leq \ell-1$ if for all sequence
$(x_1,\ldots,x_\ell)$ in $R$, one can find $a_1,\ldots,a_\ell\in R$
and  $m_1,\ldots,m_\ell\in \N$ such that
$$ x_1^{m_1}(\cdots(x_\ell^{m_\ell}(1+a_\ell x_\ell)+\cdots)+a_1x_1)=   
0
$$

In particular a ring is of dimension $\leq 0$ \ssi for all $x\in R$
there exists $n\in\N$ and $a\in R$ such that $x^n= ax^{n+1}$. And it is 
of
dimension $<1$ \ssi it is absurd to find $x\in R$ such that, for all
 $n\in\N$ and all $a\in R$, $x^n\neq ax^{n+1}$.

 Notice that the ring of real numbers is a local ring of dimension
 $<1$, but it cannot be proved constructively to be of dimension $\leq 
0$.

Notice also that a local ring  is of dimension $\leq 0$  \ssi
$$\forall x\in A\; \; \; x \; \; {\rm is\; invertible\; or\;
nilpotent}
$$
\subsubsection*{Krull dimension of a polynomial ring over a discrete 
field}
First we have.
\begin{proposition}
\label{propKrDimetDegTr}
Let $K$ be a discrete field, $R$ a commutative $K$-algebra, and $x_1$,
\ldots, $x_\ell$ in $R$ algebraically dependent over $K$. The sequence
$(x_1,\ldots,x_\ell)$ is pseudo singular.
\end{proposition}
\begin{proof}{Proof}
Let $Q(x_1,\ldots,x_\ell)= 0$ be a algebraic dependence relation
over $K$. Let us order the non zero monomials of $Q$ along
the lexicographic ordering. We can suppose that the coefficient of
the first monomial is
$1$. Let
$x_1^{m_1}x_2^{m_2}\cdots x_\ell^{m_\ell}$ be this momial,
it is clear that
$Q$ can be written on the form
$$ Q= x_1^{m_1}\cdots x_\ell^{m_\ell}+
x_1^{m_1}\cdots x_\ell^{1+m_\ell}R_\ell+
x_1^{m_1}\cdots x_{\ell-1}^{1+m_{\ell-1}}R_{\ell-1}+\cdots+
x_1^{m_1}x_2^{1+m_2}R_2+ x_1^{1+m_1}R_1
$$
and this is the desired collapsus.
\end{proof}

It follows that we have:
\begin{theorem}
\label{thKDP} Let $K$ be a discrete field. The Krull dimension of the 
ring
$K[X_1,\ldots,X_\ell]$ is equal to $\ell$.
\end{theorem}
\begin{proof}{Proof}
Given proposition \ref{propKrDimetDegTr} it is enough to check that the
sequence $(X_1,\ldots,X_\ell)$ is pseudo regular. But this sequence is
regular.
\end{proof}

Notice that we got this basic result quite directly, and that our
argument is of course also valid classically (with the usual
definition of Krull dimension).
This contradicts the current opinion that constructive arguments are
necessarily more involved than classical proofs.
\subsection{Krull dimension and local-global principle}
\label{subsecKrLocGlob}
\subsubsection*{Comaximal monoids}
\begin{definition}
\label{def.moco} ~
\begin{itemize}
\item [$(1)$] The monoids $S_1,\ldots ,S_n$ of a ring $R$ are
 {\em \comz} \ssi an ideal of $R$ that meets each $S_i$ contains
 $1$, \cad
$$ \forall s_1\in S_1 \;\cdots\; \forall s_n\in S_n \;\;
\exists a_1,\ldots, a_n\in A\quad \sum_{i= 1}^{n} a_i s_i = 1.
$$
\item [$(2)$] The
{\em \mos $S_1,\ldots ,S_n$ of the ring $R$ cover the \mo $S$}
if $S$ is a subset of each $S_i$ and if any ideal of $R$ which meets
each of the $S_i$ meets also $S$, \cad
$$ \forall s_1\in S_1 \;\cdots\; \forall s_n\in S_n \;\;
\exists a_1,\ldots, a_n\in A\quad \sum_{i= 1}^{n} a_i s_i \in S.
$$
\end{itemize}
\end{definition}

\begin{notation}
\label{notaS}
{\rm  If $(I;U)$ is an \proi of $R$, we write $\cS(I;U)$ the \mo
$\cM(U)+\gen{I}$ of the \proi obtained by completing $(I;U)$.
}
\end{notation}

The fundamental example of \moco is the following:
when $s_1,\ldots, s_n\in R$ are such that
$\gen{s_1,\ldots, s_n}= \gen{1}$,  the \mos $\cM(s_i)$ are \comz.

The following two lemmas are also quite usefull to build
\mocoz.
\begin{lemma}
\label{lemAssoc} (easy computations)
\begin{itemize}
\item [$(1)$] (associativity) If the \mos $S_1,\ldots ,S_n$ of the ring
$R$ cover the \mo $S$  and if each $S_\ell$ is covered by the \mos
$S_{\ell,1},\ldots ,S_{\ell,m_\ell}$ then the \mos $S_{\ell,j}$ cover 
$S$.
\item [$(2)$] (transitivity) Let $S$ be a \mo of the ring $R$ and
$S_1,\ldots ,S_n$ be \moco of the localization $R_S$.
For $\ell= 1,\ldots,n$ let $V_\ell$ be the  \mo
of  $R$ which consist of the numerators of the elements of $S_\ell$.
The \mos $V_1,\ldots ,V_n$ cover  $S$.
\end{itemize}
\end{lemma}

\begin{lemma}
\label{lemRecouvre}
Let $U$ and $I$ be two subsets of the ring $R$ and $a\in R$, then the 
\mos
$\cS(I,a;U)$ and $\cS(I;a,U)$ cover the \mo $\cS(I;U)$.
\end{lemma}
\begin{proof}{Proof}
For  $x\in \cS(I;U,a)$ and  $y\in \cS(I,a;U)$
we must find a  \coli $x_1x+y_1y\in \cS(I;U)$ ($x_1,y_1\in R$).
We write $x= u_1a^k+j_1$, $y= (u_2+j_2)-(az)$ with $u_1,u_2\in 
\cM(U)$,
   $j_1,j_2\in \cI(I)$,   $z\in R$. The fundamental identity
$\; c^k-d^k= (c-d)\times \cdots\; $ gives
$y_2\in R$  such that $y_2y= (u_2+j_2)^k-(az)^k= (u_3+j_3)-(az)^k$
and we write $z^kx+u_1y_2y= u_1u_3+u_1j_3+j_1z^k= u_4+j_4$.
\end{proof}

\begin{corollary}
\label{corS} Let $u_1,\ldots,u_n\in R$.
Let $S_k= \cS((u_i)_{i>k};u_k)$ $(k= 1,\ldots,n)$,
$S_0= \cS((u_i)_{i= 1,\ldots,n};1)$. Then the \mos
$S_0,S_1,\ldots,S_n$ are \comz.
\end{corollary}

The \moco are a constructive tool which allows in general
to replace abstract local-global arguments
by explicit computations.
If  $S_1,\ldots ,S_n$ are \moco of the ring $R$, the product of all
localisation $R_{S_i}$ is a faithfully flat $R$-algebra.
Hence a lot of properties are true for $R$ \ssi they hold
for each of the $R_{S_i}$.

 In the next paragraph this will be illustrated on the example
of Krull dimension.
\subsubsection*{Local character of Krull dimension}
The following proposition is direct.
\begin{proposition}
\label{propKrLoc}
Let $R$ be a ring. Its Krull dimension is always greater or equal
to any of its quotient or localisation.
More precisely, any \proel which collapses in $R$ collapses
in any quotient and localisation of $R$ and any \proel
in a localisation of $R$ is equivalent to an \proel
of $R$.
Finally, if an \proel $\ov{(a_1,\ldots,a_\ell)}$ of $R$ collapses in
a localisation $R_S$, there exists $m$ in $S$ such that
$\ov{(a_1,\ldots,a_\ell)}$
collapses in  $R[1/m]$.
\end{proposition}

\begin{proposition}
\label{propKrLocGlob}
Let  $S_1,\ldots ,S_n$  be \moco  of the ring $R$ and $\cC$ be an \proc
of $R$. Then $\cC$ collapses in $R$ \ssi it  collapses in each of the
 $R_{S_i}$. In particular the Krull dimension of $R$ is $\leq \ell$ \ssi
the Krull dimension of each of the $R_{S_i}$ is $\leq \ell$.
\end{proposition}
\begin{proof}{Proof}
We have to show that an \proc $\cC$ collapses in $R$ if it  collapses in 
each
of the $R_{S_i}$. To simplify let us take a chain of length
$2$: $((J_0,U_0),(J_1,U_1),(J_2,U_2))$ with ideals $J_k$  and \mos
$U_k$. In each  $R_{S_i}$ we have an equality
$$ u_{0,i}\,u_{1,i}\,u_{2,i}+u_{0,i}\,u_{1,i}\,j_{2,i}
+u_{0,i}\,j_{1,i}+j_{0,i}= 0
$$
with $u_{k,i}\in U_k$ and  $j_{k,i}\in J_kR_{S_i}$.
This implies an equality in $R$ of the form
$$ s_i\,u_{0,i}\,u_{1,i}\,u_{2,i}+u_{0,i}\,u_{1,i}\,j'_{2,i}
+u_{0,i}\,j'_{1,i}+j'_{0,i}= 0
$$
with $s_i\in S_i$, $u_{k,i}\in U_k$ and  $j'_{k,i}\in J_k$.
We take $u_k= \prod_iu_{k,i}$. By multiplying the previous equation
by a suitable product we get an equality
$$ s_i\,u_{0}\,u_{1}\,u_{2}+u_{0}\,u_{1}\,j''_{2,i}
+u_{0}\,j''_{1,i}+j''_{0,i}= 0\quad \quad (E_{i})
$$
with $s_i\in S_i$, $u_{k}\in U_k$ and $j''_{k,i}\in J_k$.
We now write
 $\sum_i a_is_i= 1$, we multiply the each equality  $(E_i)$
by $a_i$ and we sum all these equalities.
\end{proof}

\subsubsection*{An application}
In classical mathematics the Krull dimension of a ring is
the upper bound of the Krull dimension of the localisation in
each maximal ideals.
This follows easily (classically) from  propositions \ref{propKrLoc} and
\ref{propKrLocGlob}.

Proposition \ref{propKrLocGlob} should have the same concrete
consequences (that we can obtain non constructively by using the 
classical
property above) even we don't have access to the maximal ideals.

\ss We will limit ourselves here to describe a simple example, where
we do have access to the maximal ideals.
Suppose that we have a simple constructive argument showing that
the Krull dimension  of  $\Z_{(p)}[x_1,\ldots,x_\ell]$ is
 $\leq\ell+1$
($p$ bing an arbitrary prime number, and $\Z_{(p)}$ the localisation of
$\Z$ in $p\Z$). We can then deduce that the same holds for
$R= \Z[x_1,\ldots,x_\ell]$ using the local-global principle above.

 Indeed, consider a sequence $(a_1,\ldots,a_{\ell+2})$ in $R$.
The collapsus of the  \proel $\ov{(a_1,\ldots,a_{\ell+2})}$ in
$\Z_{(2)}[x_1,\ldots,x_\ell]$ can be read as a collapsus in
$\Z[1/m_0][x_1,\ldots,x_\ell]$ for some odd $m_0$.
For each of the prime divisor  $p_i$ of $m$ ($i= 1,\ldots,k$), the
collapsus of the \proel $\ov{(a_1,\ldots,a_{\ell+2})}$ in
$\Z_{(p_i)}$ can be read as a collapsus in
$\Z[1/m_i][x_1,\ldots,x_\ell]$ for some $m_i$ relatively prime to $p_i$.
The integers $m_i$ ($i= 0,\ldots,k$) generate the ideal $\gen{1}$, 
hence the
\mos $\cM(m_i)$ are \com and we can apply proposition
\ref{propKrLocGlob}.


\section{Distributive lattice, Entailment relations and Krull dimension}
\label{secKrullTreil}
\subsection{Distributive lattices, filters and spectrum}
\label{subsecTrd1}
A distributive lattice is an ordered set with finite sups and infs,
a minimum element (written $0$) and a maximum element (written $1$).
The operations sup and inf are supposed to be distributive w.r.t.
the other. We write these operations
$\vu$ and $\vi$. The relation
 $a\leq b$  can then be defined by $a\vu b= b$. The theory of 
distributive
lattices is then purely equational. It makes sense then to talk of
distributive lattices defined by generators and relations.

 A quite important rule, the {\em cut rule}, is the following
$$ \left( ((x\vi a)\; \leq\;  b)\quad\&\quad  (a\; \leq\; (x\vu  b))
\right)\; \Longrightarrow \; a \leq\;  b.
$$
In order to prove this, write $ x\vi a\vi b= x\vi a$  and
$a=  a\vi(x\vu b)$ hence
$$ a= (a\vi x)\vu(a\vi b)= (a\vi x\vi b)\vu(a\vi b)= a\vi b.
$$

A totally ordered set is a distributive lattice as soon as it has
a maximum and a minimum element.
We write ${\bf n}$ the totally ordered set with
 $n$ elements (this is a distributive lattice for $n\neq 0$.)
A product of distributive lattices is a distributive lattice.
Natural numbers with the divisibility relation form a distributive 
lattice
(with minimum element $1$ and maximum element $0$).
If $T$  and $T'$ are two distributive lattices, the set $\Hom(T,T')$ of 
all
morphisms
(\cad maps preserving sup, inf, $0$ and $1$) from  $T$ to $T'$ has a
natural order
given by
$$ \varphi \leq \psi \equidef  \forall x\in T\;
\;
\varphi(x) \leq \psi(x).
$$
A map between two totally ordered distributive lattices $T$ and $S$
is a morphism \ssi it is nondecreasing and $0_T$ and $1_T$ are mapped 
into
$0_S$ and $1_S$.

The following proposition is direct.
\begin{proposition}
\label{propIdeal} Let $T$ be a distributive lattice and $J$ a subset of
$T$. We consider the distributive lattice $T'$ generated by $T$
and the relations
 $x= 0$ for $x\in J$ ($T'$ is a quotient of $T$).  Then
\begin{itemize}
\item  the equivalence class of $0$ is the set of $a$ such that
for some finite subset $J_0$ of $J$:
$$  a\; \leq\; \Vu_{x\in J_0}x\quad {\rm in} \;  T
$$
%
\item  the equivalence class of $1$ is the set of $b$ such that
for some finite subset $J_0$ of $J$:
$$1\;= \;\left( b\;\vu\;\Vu_{x\in J_0}x\right)\quad {\rm in} \;  T
$$
%
\item  More generally $a\leq_{T'}b$ \ssi
for some finite subset $J_0$ of $J$:
$$  a\; \leq\;  \left( b\; \vu\; \Vu_{x\in J_0}x\right)
$$
\end{itemize}
\end{proposition}

In the previous proposition, the equivalence class of $0$ is called
an {\em ideal} of the lattice; it is the ideal generated by $J$.
We write it $\gen{J}_T$. We can easily check that
an ideal $I$ is a subset such that:
$$\begin{array}{rcl}
  & &  0 \in I   \\
x,y\in I& \Longrightarrow   &  x\vu y \in I   \\
x\in I,\; z\in T& \Longrightarrow   &  x\vi z \in I   \\
\end{array}$$
(the last condition can be written $(x\in I,\;y\leq x)\Rightarrow y\in 
I$).

Furthermore, for any morphim
 $\varphi :T_1\rightarrow T_2$,
$\varphi^{-1}(0)$ is an ideal of $T_1$.

A  {\em principal ideal} is an ideal generated by one element $a$.
We have $\gen{a}_T= \{x\in T\; ;\; x\leq a \}$. Any finitely generated
ideal is principal.

 The dual notion of ideal is the one of {\em filter}.
A filter
$F$ is the inverse image of $1$
by a morphism. This is a subset such that:
$$\begin{array}{rcl}
  & &  1 \in F   \\
x,y\in F& \Longrightarrow   &  x\vi y \in F   \\
x\in F,\; z\in T& \Longrightarrow   &  x\vu z \in F   \\
\end{array}$$
\begin{notation}
{\rm We write $\Pf(X)$ the set of all finite subsets of the set $X$.
If $A$ is a finite subset of a distributive lattice $T$
$$ \Vu A:= \Vu_{x\in A}x\qquad {\rm and}\qquad \Vi A:= \Vi_{x\in A}x
$$
We write $A \vdash B$ or $A \vdash_T B$ the relation defined on the set
 $\Pf(T)$:
$$ A \vda B \; \; \equidef\; \; \Vi A\;\leq \;
\Vu B
$$
}
\end{notation}
Note the relation  $A \vdash B$ is well defined on finite subsets 
because of
associativity commutativity and idempotence of the operations
 $\vi$  and $\vu$.
Note also
$\; \emptyset  \vda \{x\}\; \Rightarrow\;  x= 1\; $ and
$ \{y\} \vda \emptyset\; \Rightarrow \; y= 0$.
This relation satisfies the following axioms, where
we write
 $x$ for $\{x\}$ and $A,B$ for $ A\cup B$.
$$\begin{array}{rcrclll}
&    & a  &\vda& a    &\; &(R)     \\
A \vda B &   \; \Longrightarrow \;  & A,A' &\vda& B,B'   &\; &(M)     \\
(A,x \vda B)\;
\&
\;(A \vda B,x)  &   \Longrightarrow  & A &\vda& B &\;
&(T)
\end{array}$$
we say that the relation is reflexive,
 \label{remotr} monotone and
transitive. The last rule is also called emph{cut rule}.
Let us also mention the two following rules of ``distributivity'':
$$\begin{array}{rcl}
(A,\;x \vda B)\;\& \;(A,\;y \vda B)  &  \;  \Longleftrightarrow  \; &
A,\;x\vu y \vda B  \\
(A\vda B,\;x )\;\&\;(A \vda B,\;y)  &   \Longleftrightarrow  &
A\vda B,\;x\vi y
\end{array}$$

The following is a corollary of proposition
\ref{propIdeal}.
\begin{proposition}
\label{propIdealFiltre} Let $T$ be a distributive lattice and
$(J,U)$ a couple of subsets of $T$.
We consider the distributive lattice $T'$ generated by $T$ and by the 
relations
$x= 0$ for $x\in J$ and $y= 1$ for $y\in U$
($T'$ is a quotient of $T$). We have that:
\begin{itemize}
\item  the equivalence class of $0$ is the set of elements $a$ such 
that:
$$ \exists J_0\in\Pf(J),\; U_0\in\Pf(U) \qquad
a,\; U_0 \; \vdash_T\;  J_0
$$
\item  the equivalence class of $1$ is the set of elements $b$ such 
that:
v\'erifient:
$$  \exists J_0\in\Pf(J),\; U_0\in\Pf(U)\qquad
 U_0 \; \vdash_T\; b,\; J_0
$$
\item  More generally $a\leq_{T'}b$ \ssi
there exists a finite subset $J_0$ of $J$ and a finite subset $U_0$ of
$U$ such that, in $T$:
$$  a,\; U_0 \; \vdash_T\; b,\; J_0
$$
\end{itemize}
\end{proposition}

We shall write $T/(J= 0,U= 1)$  the quotient lattice $T'$ described in
proposition \ref{propIdealFiltre}. Let $\psi:T\rightarrow T'$ be the
canonical surjection. If  $I$ is the ideal $\psi^{-1}(0)$ and $F$ the
filter $\psi^{-1}(1)$, we say that the {\em ideal $I$ and the filter $F$
are conjugate}. By the previous proposition, an ideal $I$ and
a filter $F$ are conjugate \ssi we have:
$$\begin{array}{cl}
 \left[ x\in T,\,I_0\in\Pf(I),\, F_0\in\Pf(F), \;
  (x,\; F_0 \vda I_0)\right]\;\Longrightarrow\;  x\in I& \quad{\rm
  and}\
\\
\left[x\in T,\,I_0\in\Pf(I),\, F_0\in\Pf(F), \;(F_0 \vda x,\;  I_0)
\right]\;\Longrightarrow\;  x\in F.
\end{array}$$
This can also be formulated as follows:
$$
(f\in F,\; x\vi f \in I) \Longrightarrow x\in I
\quad {\rm and}\quad
(j\in I,\; x\vu j \in F) \Longrightarrow x\in F.
$$
When an ideal $I$ and a filter $F$ are conjugate, we have
$$
1\in I\; \;\Longleftrightarrow\;\;  0\in F
\;\; \Longleftrightarrow\;\;  (I,F)= (T,T).
$$
We shall also write
$T'=  T/(J= 0,U= 1)$ as $T/(I,F)$.
By proposition
 \ref{propIdealFiltre}, an \homo $\varphi$  from
$T$ to another lattice $T_1$ satisfying
$\varphi(J)= \{0\}$ and $\varphi(U)= \{1\}$ can be factorised in an
unique way through the quotient $T'$.

As shown by the example of totally ordered sets
a quotient of distributive lattices is not in general
characterised by the equivalence classes of $0$ and $1$.

\ms Classically a {\em prime ideal} $I$ of a lattice
is an ideal whose complement $F$ is a filter (which is then
a {\em prime filter}). This can be expressed by
$$ 1\notin I\qquad {\rm and}\qquad (x\vi y)\in I\; \Longrightarrow \;
(x\in I{\rm \; or\; }y\in I)\qquad\qquad(*)
$$
which can also be expressed by saying that $I$ is the kernel
of a morphism from $T$ into the lattice with two elements
written $\Deux$.
Constructively it seems natural to take the definition $(*)$,
where ``or'' is used constructively. The notion of prime filter
is then defined
 in a dual way.

 The {\em spectrum} of the lattice $T$, written  $\Spec(T)$ is defined
as the set $\Hom(T,\Deux)$. It is isomorphic to the ordered set of
all detachables prime ideals. The order relation is then
reverse inclusion.
We have $\Spec(\Deux)\simeq \Un$,  $\Spec(\Trois)\simeq \Deux$,
$\Spec(\Quatre)\simeq \Trois$, etc\ldots

\begin{definition}
\label{defProiT} Let $T$ be a distributive lattice.
\begin{itemize}
\item An {\em \proiz} in  $T$ is given by a pair $(J,U)$ of
subsets of $T$. We consider this as an incomplete specification
for a prime ideal $P$ satisfying $J\subseteq P$ and $U\cap 
P= \emptyset$.
It is {\em finite} iff $J$ and $U$ are finite, and {\em trivial}
iff $J= U= T$.
\item  An \proi $(J,U)$  is {\em saturated} iff $J$ is an ideal, $U$
a filter and $J$ and $U$ are conjugate. Any \proi generates a saturated 
\proi
$(I,F)$ as described in proposition \ref{propIdealFiltre}.
\item We say that the \proi $(J,U)$ {\em collapses}
iff the saturated \proi $(I,F)$ it generates is trivial.
This means that the quotient lattice $T'= T/(J= 0,U= 1)$ is a 
singleton
\cad  $1 \leq_{T'}0$, which means that there is a finite subset
$J_0$ of $J$ and a finite subset $U_0$ of $U$ such that
$$ U_0 \vda J_0.
$$
\end{itemize}
\end{definition}
We have the following theorem, similar to theorem \ref{ThColSim1}.

\begin{theorem}
\label{lemColSimT1} {\em (Simultaneous collapse for \proisz)} Let
$(J,U)$ be an \proi for a lattice $T$ and $x$ be an element of $T$.
\begin{itemize}
\item [$(1)$] If the \prois $(J\cup\{x\},U)$ and
$(J,U\cup\{x\})$  collapse, then so does $(J,U)$.
\item [$(2)$] The \proi  $(J,U)$ generates a minimum saturated \proiz.
We get it by adding in
$U$ (resp. $J$) any $x\in A$ such that the \proi
 $(J\cup\{x\},U)$
(resp. $(J,U\cup\{x\})$) collapses.
\end{itemize}
\end{theorem}
\begin{proof}{Proof} Let us prove (1).
  We have two finite subsets $J_0,J_1$ of $J$
and two finite subsets $U_0,U_1$ of $U$ such that
$$ x,\; U_0 \vda J_0\quad{\rm  and}\quad   U_1  \vda x,\; J_1
$$
donc
$$ x,\; U_0,\; U_1 \vda J_0,\; J_1\quad{\rm  and}\quad  U_0,\; U_1 \vda
x,\;J_0,\; J_1
$$
Hence by the cut rule
$$  U_0,\; U_1 \vda J_0,\; J_1
$$
The point (2) has already been proved (in a slightly different 
formulation)
in proposition \ref{propIdealFiltre}.
\end{proof}

Notice the crucial role of the cut rule.

We deduce the following proposition.
\begin{proposition}
\label{propTr2} \Tcgi
If $(J,U)$ is an \proi which does not collapse
then there exists
$\varphi\in\Spec(T)$
such that $J\subseteq \varphi^{-1}(0)$  and
$U\subseteq \varphi^{-1}(1)$.
In particular if $a\not\leq b$, there exists $\varphi\in\Spec(T)$ such 
that
$\varphi(a)= 1$ and $\varphi(b)= 0$. Also, if $T\neq \Un$,
$\Spec(T)$ is non empty.
\end{proposition}

A corollary is the following representation theorem
(Birkhoff theorem)
\begin{theorem}
\label{thRep} {\em (Representation theorem)}
\Tcgi The map
$\theta_T: T\rightarrow \cP(\Spec(T))$ defined by
$a\mapsto \left\{\varphi\in \Spec(T)\; ;\; \varphi(a)= 1 \right\}$
is an injective map of distributive lattice. This means that any 
distributive
lattice can be represented as a lattice of subsets of a set.
\end{theorem}

Another corollary is the following proposition.
\begin{proposition}
\label{propRep2}
\Tcgi Let $\varphi:T\rightarrow T'$ a map of \trdisz; $\varphi$
is injective \ssi $\Spec(\varphi):\Spec(T')\rightarrow \Spec(T)$ is
surjective.
\end{proposition}
\begin{proof}{Proof}
We have the equivalence
$$ a\not=  b\quad \Longleftrightarrow\quad  a\vi b\not=  a\vu b\quad
\Longleftrightarrow\quad
a\vu b\not\leq a\vi b
$$
Assume that $\Spec(\varphi)$ is surjective. If $a\not=  b$ in $T$, take
$a'= \varphi(a)$, $b'= \varphi(b)$ and let
  $\psi\in\Spec(T)$ be such that $\psi(a\vu b)= 1$ and
 $\psi(a\vi b)= 0$. Since  $\Spec(\varphi)$ is surjective there exists
$\psi'\in\Spec(T')$ such that $\psi= \psi'\varphi$ hence  $\psi'(a'\vu 
b')= 1$ is
 $\psi'(a'\vi b')= 0$, hence $a'\vu b'\not\leq a'\vi b'$ and $ 
a'\not=  b'$.\\
Suppose that  $\varphi$ is injective. We identify $T$ to a sublattice of
$T'$. If $\psi\in\Spec(T)$, take $I= \psi^{-1}(0)$ and $F= \psi^{-
1}(1)$.
Then $(I,F)$ cannot collapse in $T'$ since it would then collapse in
dans $T$. Hence there exists $\psi'\in\Spec(T')$ such that $\psi'(I)=  
0$ and
 $\psi'(F)=  1$, which means $\psi= \psi'\varphi$.
\end{proof}

Of course, these three last results are hard to interpret in
a computational way. An intuitive interpretation is that we
can proceed ``as if'' any distributive lattice is a lattice of
subsets of a set. The goal of Hilbert's program is to give
a precise meaning to this sentence, and explain  what is meant
by ``as if'' there.

\subsection{Distributive lattices and \entrels}
\label{subsecTrd2}
An interesting way to analyse the description of distributive lattices
defined by generators and relations is to consider the relation
$A \vda B$ defined on the set $\Pf(T)$ of finite subsets of a
lattice $T$.
Indeed if  $S\subseteq T$ generates the lattice $T$,
 then the relation  $\vda$ on $\Pf(S)$ is enough to characterise
the lattice $T$, because any formula on $S$ can be rewritten,
in normal conjunctive form (inf of sups in $S$) and normal
disjonctive form (sup of infs in $S$). Hence if we want to compare
two elements of the lattice generated by $S$ we write the first
in normal disjunctive form, the second in normal conjunctive form,
and we notice that
$$ \Vu_{i\in I}\left(\Vi A_i \right)\; \leq \; \Vi_{j\in J}\left(\Vu B_j
\right)
\qquad \Longleftrightarrow\qquad  \&_{(i,j)\in I\times J}\;  \left(
A_i \vda  B_j\right)
$$

\begin{definition}
\label{defEntrel}
For an arbitrary set $S$, a relation over  $\Pf(S)$ which is
reflexive, monotone and transitive (see page  \pageref{remotr}) is
called an {\em entailment relation.}
\end{definition}

The notion of \entrels goes back to Gentzen sequent calculus, where
the rule $(T)$ (the cut rule) is first explicitly stated, and
plays a key role. The connection with distributive lattices has been
emphasized in \cite{cc,cp}.
The following result (cf. \cite{cc}) is fundamental. It says that the
three properties of entailment relations are exactly the ones needed
in order to have a faithfull interpretation in distributive lattices.

\begin{theorem}
\label{thEntRel1} {\rm  (fundamental theorem of \entrelsz)} Let $S$ be a 
set
with an entailment relation
$\vdash_S$ over $\Pf(S)$. Let $T$ be the lattice defined by generators
and relations as follows: the generators are the elements of $S$ and the
relations are
$$ A\; \vdash_T \;  B
$$
whenever $A\; \vdash_S \; B$. For any finite subsets $A$ and $B$ of $S$ 
we have
$$  A\; \vdash_T \;  B
\; \Longrightarrow \; A\; \vdash_S \;  B.
$$
\end{theorem}
\begin{proof}{Proof}
We give an explicit possible description of the lattice $T$. The 
elements
of $T$ are represented by finite sets of finite sets of elements of $S$
$$X= \{A_1,\dots,A_n\}$$
(intuitively $X$ represents $\Vi A_1\vu\cdots\vu\Vi A_n$).
We define then inductively the relation
 $A\prec Y$ with $A\in \Pf(S)$ and $Y\in T$
(intuitively $\Vi A\leq \Vu_{C\in Y} \left(\Vi C\right) $)
\begin{itemize}
\item if $B\in Y$ and $B\subseteq A$ then $A\prec Y$
\item if $A\vdash_S y_1,\dots,y_m$ and $A,y_j\prec Y$ for
$j= 1,\ldots,m$ then $A\prec Y$
\end{itemize}
It is easy to show that if
 $A\prec Y$ and $A\subseteq A'$ then we have
also $A'\prec Y.$ It follows that  $A\prec Z$ holds
whenever $A\prec Y$ and $B\prec Z$
for all $B\in Y$. We can then define $X\leq Y$ by
$A\prec Y$ for all $A\in X$ and one can then check that
$T$ is a distributive lattice{\footnote{~$T$ is actually the quotient of
$\Pf(\Pf(S))$  by the equivalence relation: $X\leq Y$ and $Y\leq X$.}}
for the operations
$$0 =  \emptyset,~~~~1 =  \{\emptyset\},~~~~~X\vee Y =  X\cup Y,~~~~~
X\wedge
Y =  \{ A \cup B~|~A\in X,~B\in Y\}.
$$
For establishing this one first show that if
$C\prec X$ and $C\prec Y$ we have
$C\prec X\vi Y$ by induction on the proofs of
$C\prec X$ and $C\prec Y$.
We notice then that if
 $A\vdash_S y_1,\dots,y_m$ and $A,y_j\vdash_S B$
for all $j$ then $A\vdash_S B$ using $m$ times the cut rule. 
It follows that if we have
$A\vdash_T B$, \cad
$A\prec \{\{b\}~|~b\in B\}$, then we have also $A\vdash_S B$.
\end{proof}

As a first application, we give the description of the Boolean
algebra generated by a distributive lattice.
A Boolean algebra can be seen as a distributive lattice with
a complement operation $x\mapsto \overline{x}$ such that
$x\vi \overline{x}= 0$ and $x\vu \overline{x}= 1$.
The application $x\mapsto \overline{x}$ is then a map from the lattice
to its dual.
\begin{proposition}
\label{propTrBoo}
Let $T$ be a distributive lattice. There exists a free
Boolean algebra generated by $T$. It can be described as
the distributive lattice generated by the set
$T_1= T\cup\overline{T}$
{\rm (}{\footnote{~$\overline{T}$  is a disjoint copy of $T$.}}{\rm )}
with the \entrel $\;\vdash_{T_1}\;$
defined as follows:
if $A,B,A',B'$  are finite subsets of $T$ we have
$$A,\overline{B}\;\vdash_{T_1}\; A',\overline{B'}\equidef A,B'\vda
A',B\quad
{\rm in} \;  T
$$
If we write $T_{Bool}$ this lattice (which is a Boolean algebra), there 
is
a natural embedding of $T_1$  in $T_{Bool}$
and the entailment relation of $T_{Bool}$
induces on  $T_1$ the relation  $\,\vdash_{T_1}\,$.
\end{proposition}
\begin{proof}{Proof}
See \cite{cc}.
\end{proof}
Notice that by theorem \ref{thEntRel1} we have
$x\;\vdash_{T}\; y$ \ssi $x\;\vdash_{T_1}\; y$ hence the canonical
map $T\rightarrow T_1$ is one-to-one and $T$
can be identified to a subset of $T_1$.

\subsection{Krull dimension of \trdisz}
\label{subsecTrd3}
To develop a suitable constructive theory of the Krull dimension
of a distributive lattice we have to find a constructive counterpart of
the notion of increasing chains of prime ideals.

One can do it along the same lines as what has been done for
commutative rings in section \ref{secKrA}, or else
use an idea due to Joyal. It consists in building an universal
lattice $\Kr_\ell(T)$ associated to $T$ such that the points of
 $\Spec(\Kr_\ell(T))$  are (in a natural way) the chains of prime
ideals of length $\ell$.
We shall present the two descriptions and establish their
equivalence.
\subsubsection*{Partially specified chains of prime ideals}
\begin{definition}
\label{defprochTreil} In a \trdi $T$
\begin{itemize}
\item  A {\em partial specification for a chain of prime ideals}
(that we shall call {\em \procz}) is defined as follows.
An \prolo $\ell$ is a list of $\ell+1$ \prois of $T$:
$\cC= ((J_0,U_0),\ldots,(J_\ell,U_\ell))$. The \proc is {\em finite}
iff all the subsets are finite. An \prolo $0$ is nothing but an
{\em \proiz}.
\item  An \proc is {\em saturated} \ssi all the $J_i$ and $U_i$ are
conjugate, and if we have furthermore
$J_i\subseteq J_{i+1}$, $U_{i+1}\subseteq U_i$
$(i= 0,\ldots,\ell-1)$.
\item  An \proc $\cC'= ((J'_0,U'_0),\ldots,(J'_\ell,U'_\ell))$
is a {\em refinenement of the \proc}
$\cC= ((J_0,U_0),\ldots,(J_\ell,U_\ell))$
\ssi
$J_k\subseteq J'_k$, $U_k\subseteq U'_k$,
\item  An \proc $\cC$ {\em collapses}
\ssi the only saturated \proc that refines $\cC$ is the trivial \proc
$((T,T),\ldots,(T,T))$.
\end{itemize}
\end{definition}
\begin{lemma}
\label{lemColTr}
An \proc $\cC= ((J_0,U_0),\ldots,(J_\ell,U_\ell))$  in which we have
$U'_h\vda J'_h$ with $U'_h\in \Pf(U_h)$ and $J'_h\in \Pf(J_h)$
(in particular if $U_h\cap J_h\not=  \emptyset$)
for some index $h$ collapses.
\end{lemma}
\begin{proof}{Proof}
Let $((I_0,F_0),\ldots,(I_\ell,F_\ell))$ be a saturated \proc which
is a refinement of $\cC$. Since the
\proi  $(I_h,F_h)$  collapses and since $I_h$ and
$F_h$ are conjugate, we have $1\in  I_h$ and  $0\in  F_h$.
For all index $ j>h$  we thus have $1\in  I_j$ and hence   $0\in  F_j$.
Similarly for all index $ j<h$ we have   $0\in  F_j$  and hence
$1\in  I_j$.
\end{proof}

In the following theorem the points (3) and (2) correspond to the point
(1) and (2) in theorem \ref{ThColSimKrA}.
\begin{theorem}
\label{thColSimT2} {\em (Simultaneous collapse for the \procs in
  \trdisz)} ~\\
Let $\cC= ((J_0,U_0),\ldots,(J_\ell,U_\ell))$ be an \proc in a \trdi
$T.$
\begin{itemize}
\item [$(1)$] The \proc $\cC$ collapses \ssi there exists
  $x_1,\ldots,x_\ell\in T$ and a finite \proc
  $\cC'= ((J'_0,U'_0),\ldots,(J'_\ell,U'_\ell))$ of which $\cC$ is a
  refinement, with the following relations in $T$ (where $\vda$ is
  the \entrel of $T$):
$$\begin{array}{rcl}
 x_1,\;   U'_0& \vda  &  J'_0  \\
 x_2,\;   U'_1& \vda  &  J'_1 ,\;  x_1  \\
\vdots\quad & \vdots  & \quad\vdots  \\
 x_\ell,\;   U'_{\ell-1}& \vda  &  J'_{\ell-1} ,\;  x_{\ell-1}  \\
  U'_\ell& \vda  &  J'_\ell ,\;  x_\ell
\end{array}$$
\item [$(2)$] The \proc $\cC$ generates a minimum saturated \procz.
We get it by adding to
$U_i$ (resp. $J_i$) each element  $a\in A$ such that the \proc
$((J_0,U_0),\ldots,(J_i\cup\{a\},U_i),\ldots,(J_\ell,U_\ell))$
(resp. $((J_0,U_0),\ldots,(J_i,U_i\cup\{a\}),\ldots,(J_\ell,U_\ell))$)
collapses.
\item [$(3)$] Take $x\in T$. Suppose that the \procs
$((J_0,U_0),\ldots,(J_i\cup\{x\},U_i),\ldots,(J_\ell,U_\ell))$
and
$((J_0,U_0),\ldots,(J_i,U_i\cup\{x\}),\ldots,(J_\ell,U_\ell))$
both collapse, then so does $\cC$.
\end{itemize}
\end{theorem}
\begin{proof}{Proof}
Let us begin with the two first points. We can always suppose the
\proc $\cC$ to be finite, for one can always deduce the general
case from this one by looking as the given \proc as an inductive
limit of all the finite \proc of which it is a refinement.
In the case where $\cC$ is finite, we can systematically replace $U'_i$  
by
$U_i$ and  $J'_i$  par $J_i$.
Let $\cC_1= ((I_0,F_0),\ldots,(I_\ell,F_\ell))$ be the \proc
defined in (2). We shall show that
\begin{itemize}
\item [$(\alpha)$] If $\cC$ satisfies the relations (1)
any saturated \proc which refines $\cC$ is trivial (\cad $\cC$
collapses).
\item [$(\beta)$]  The \proc $\cC_1$ is saturated.
\item [$(\gamma)$] Any saturated \proc which refines $\cC$ also refines
$\cC_1$.
\item [$(\delta)$] If $\cC_1$ is trivial, $\cC$ satisfies the relations 
(1).
\end{itemize}
This will establish (1) and (2).
$(\alpha)$
Let $((I'_0,F'_0),\ldots,(I'_\ell,F'_\ell))$ be a saturated \proc which
refines $\cC$. We consider the relations (1)
$$\begin{array}{rcl}
 x_1,\; U_0& \vda  & J_0  \\
 x_2,\; U_1& \vda  & J_1 ,\; x_1  \\
\vdots\qquad & \vdots  & \qquad\vdots  \\
 x_\ell,\; U_{\ell-1}& \vda  & J_{\ell-1} ,\; x_{\ell-1}  \\
 U_\ell& \vda  & J_\ell ,\; x_\ell
\end{array}$$
Since  $I'_0$ and $F'_0$ are conjugate, the first of these
relations gives $x_1\in I'_0$.
Hence  $x_1\in I'_1$,
and the second relation gives $x_2\in I'_1$.
Going on in this way we get for the last relation
$U_\ell \vda   J_\ell ,\; x_\ell$. avec $x_\ell\in I'_\ell$,
which furnishes the desired collapsus (lemma \ref{lemColTr}).
\\
$(\beta)$
We give the proof for $\ell= 3$.
We first show that the $I_j$ are ideals.
We give the proof for $j= 1$.
In order to show $0\in I_1$ take
$x_1= 0,\; x_2= x_3= 1$. Similarly to show $J_1\subseteq  I_1$ take
$x\in J_1$ and $x_1= 0,\; x_2= x_3= 1$.
That $x\in I_1$  and $y\leq x$ imply $y\in I_1$ is
immediate: we can keep the same  $x_i$. Suppose now $x,y\in I_1$
and let us show
$x\vu y\in I_1$. We have by hypothesis some $x_i$'s  and $y_i$'s
satisfying the following relations
$$\begin{array}{rclcrcl}
 x_1,\; U_0& \vda  & J_0  &\qquad &
 y_1,\; U_0& \vda  & J_0  \\
 x_2,\; U_1,\; x& \vda  & J_1 ,\; x_1  &&
 y_2,\; U_1,\; y& \vda  & J_1 ,\; y_1  \\
 x_3,\; U_{2}& \vda  & J_{2} ,\; x_{2}  &&
 y_3,\; U_{2}& \vda  & J_{2} ,\; y_{2}  \\
 U_3& \vda  & J_3 ,\; x_3  &&
 U_3& \vda  & J_3 ,\; y_3
\end{array}$$
Using distributivity, we get
$$\begin{array}{rclcrcl}
 (x_1\vu y_1),\; U_0& \vda  & J_0   \\
 (x_2\vi y_2),\; U_1,\; (x\vu y)& \vda &J_1,\;(x_1\vu y_1)\\
 (x_3\vi y_3),\; U_{2}& \vda  & J_{2} ,\; (x_2\vi y_2)  \\
 U_3& \vda  & J_3 ,\; (x_3\vi y_3)  \\
\end{array}$$
Let us show now that the corresponding ideals and filters are
conjugate, for instance that $I_1$ and $F_1$ are conjugate.
We assume  $x\vi y\in I_1,\; y\in F_1$ and we show
 $x\in I_1$.
We have by hypothesis some $x_i$'s  and $y_i$'s
satisfying the following relations
$$\begin{array}{rclcrcl}
 x_1,\; U_0& \vda  & J_0  &\qquad &
 y_1,\; U_0& \vda  & J_0  \\
 x_2,\; U_1,\; (x \vi y)& \vda  & J_1 ,\; x_1  &&
 y_2,\; U_1& \vda  & J_1 ,\; y_1 ,\; y \\
 x_3,\; U_{2}& \vda  & J_{2} ,\; x_{2}  &&
 y_3,\; U_{2}& \vda  & J_{2} ,\; y_{2}  \\
 U_3& \vda  & J_3 ,\; x_3  &&
 U_3& \vda  & J_3 ,\; y_3
\end{array}$$
Using distributivity, we get
$$\begin{array}{rclcrcl}
 (x_1\vu y_1),\; U_0& \vda  & J_0   \\
 (x_2\vi y_2),\; U_1,\; x,\; y& \vda&J_1,\;(x_1\vu y_1)\\
 (x_2\vi y_2),\; U_1,\; x& \vda&J_1,\;(x_1\vu y_1),\; y\\
 (x_3\vi y_3),\; U_{2}& \vda  & J_{2} ,\; (x_2\vi y_2)  \\
 U_3& \vda  & J_3 ,\; (x_3\vi y_3)
\end{array}$$
The relations \num 2 and 3 give by cut
$$\begin{array}{rclcrcl}
 (x_2\vi y_2),\; U_1,\; x& \vda &J_1,\;(x_1\vu y_1)\\
\end{array}$$
The proof is finished.
\\
$(\gamma)$
We give the proof for $\ell= 3$.
Let $((I'_0,F'_0),\ldots,(I'_3,F'_3))$ be a saturated \proc
which refines $\cC$.Let us show $I_1\subseteq I'_1$.
Take $x\in I_1$, we have
$$\begin{array}{rclcrcl}
 x_1,\; U_0& \vda  & J_0   \\
 x,\; x_2,\; U_1& \vda  & J_1 ,\; x_1  \\
 x_3,\; U_{2}& \vda  & J_2 ,\; x_2  \\
 U_3& \vda  & J_3 ,\; x_3
\end{array}$$
We deduce from this successively $x_1\in I'_0\subseteq I'_1$,
 $x_3\in F'_3\subseteq F'_2$,
$x_2\in F'_2\subseteq F'_1$,  and finally
$x\in I'_1$. Notice that the proof of the point $(\alpha)$
can be seen as a particular case of the proof of the point $(\gamma)$.
\\
$(\delta)$    is direct.

\sni Finally we prove (3). We have $x\in I_i$  and  $x\in F_i$,  and 
hence
$\cC_1$ collapses (lemma \ref{lemColTr}). Hence  $\cC$ collapses.
\end{proof}

\begin{definition}
\label{defproch3T} ~
\begin{itemize}
\item Two \procs that generate the same saturated \proc are {\em
equivalent}.
\item An \proc {\em of finite type} is one
which is equivalent to a finite \procz.
\item An \proc is {\em strict} \ssi we have
$V_i\cap I_{i+1}\neq \emptyset$  $(i= 0,\ldots,\ell-1)$ in its
generated saturated \procz.
\item  A saturated \proc $\cC= ((J_0,U_0),\ldots,(J_\ell,U_\ell))$ is
{\em frozen} \ssi it does not collapse and if we have $J_i\cup U_i= T$ 
for
$i= 0,\ldots,\ell$. An \proc is frozen iff its saturation is.
To give a strict frozen \proc is the same as to give a
strictly increasing chain of detachable prime ideals.
\end{itemize}
\end{definition}

We think of an \prolo $\ell$ as a partial specification of
an increasing chains of prime ideals $P_0,\ldots,P_\ell$
such that $J_i\subseteq P_i$, $U_i\cap P_i= \emptyset$,
$(i= 0,\ldots,\ell)$.

  From the simultaneous collapse theorem we deduce the
following result which justifies this idea of
partial specification.
\begin{theorem}
\label{th.nstformelTreil} {\em (formal \nst for chains of prime ideals)}
\Tcgi Let $T$ be a  \trdi and
$((J_0,U_0),\ldots,(J_\ell,U_\ell))$ be an \proc in $T$. \Propeq
\begin{itemize}
\item [$(a)$] There exist $\ell+1$ prime ideals $P_0\subseteq
\cdots\subseteq
P_\ell$ such that
 $J_i\subseteq P_i$, $U_i\cap P_i= \emptyset $, $(i= 0,\ldots,\ell)$.
\item [$(b)$] The \proc does not collapse.
\end{itemize}
\end{theorem}
The proof is the same as the proof of theorem \ref{th.nstformel}.

\subsubsection*{Joyal's Theory}
\label{subsubsecJoyal}
The idea of Joyal is to introduce a lattice $\Kr_\ell(T)$ associated to 
$T$
by an universal condition such that the points of $\Spec(\Kr_\ell(T))$  
are
(in a natural way) the chains of prime ideals of length  $\ell$.
To give such a chain is equivalent to give morphisms
$\mu_0\geq \mu_1\geq \cdots\geq  \mu_\ell$ from $T$ to $\Deux$.
If we have a \trdi $K$ and $\ell+1$ \homos
$\varphi_0\geq \varphi_1\geq \cdots\geq  \varphi_\ell$ from $T$ to $K$ 
such
that for all lattices $ T'$ and all $\psi_0\geq \psi_1\geq \cdots\geq
\psi_\ell\in\Hom(T,T')$
we have a unique \homo  $\eta :K\rightarrow T'$  such that $\eta
\varphi_0= \psi_0$,
$\eta \varphi_1= \psi_1,$ $\ldots$,  $\eta \varphi_\ell= \psi_\ell$, 
then
the elements of
$\Spec(K)$ can be identified canonically with chains of prime ideals
of length $\ell$ in $T$.

The advantage is that $K$ is an object that we can build effectively
from $T$, in opposition to the chain of prime ideals or points in the 
spectrum.

The fact that such an universal object $\Kr_\ell(T)$ always exists and
is unique follows, constructively, from general abstract algebra 
arguments.

 The explicit description of $\Kr_\ell(T)$ is simplified by the
notion of entailment relations (\cite{cc}). More precisely we have the
following result.

\begin{theorem}
\label{thKrJoy}
Let $T$ be a \trdiz. We consider the following universal problem,
called here ``Krull problem'': to find a \trdi $K$ and $\ell+1$ \homos
$\varphi_0\geq \varphi_1\geq \cdots\geq \varphi_\ell$ from $T$ to $K$
such that, for any lattice $ T'$ and any morphism $\psi_0\geq
\psi_1\geq \cdots\geq \psi_\ell\in\Hom(T,T')$ we have one and only one
morphsim $\eta :K\rightarrow T'$ such that $\eta \varphi_0= \psi_0$,
$\eta \varphi_1= \psi_1,$ $\ldots$, $\eta \varphi_\ell= \psi_\ell$.
This universal problem admits a unique solution (up to isomorphism).
We write $\Kr_\ell(T)$ the corresponing distributive lattice.  It can
be described as the lattice generated by the disjoint union $S$ of
$\ell+1$ copies of $T$ (we shall write $\varphi_i$ the bijection
between $T$ and the $i$th copy) with the \entrel $\vdash_S$ defined
as follows. If $U_i$ and $J_i$ $(i= 0,\ldots,\ell)$ are finite subsets
of $T$ we have
$$ \varphi_0(U_0),\ldots,\varphi_\ell(U_\ell) \,\vdash_S\,
\varphi_0(J_0),\ldots,\varphi_\ell(J_\ell)
$$
\ssi there exist $x_1,\ldots,x_\ell\in T$ such that
 (where $\vda$ is the \entrel of $T$):
$$\begin{array}{rcl}
x_1,\;  U_0& \vda  &  J_0  \\
x_2,\;  U_1& \vda  &  J_1 ,\; x_1  \\
    \vdots\quad & \vdots  & \quad\vdots  \\
x_\ell,\;  U_{\ell-1}& \vda& J_{\ell-1} ,\; x_{\ell-1}  \\
 U_\ell& \vda  &  J_\ell ,\; x_\ell  \\
\end{array}$$
\end{theorem}
\begin{proof}{Proof}
First we show that the relation $\vdash_S$ on $\Pf(S)$ described in the
statement of the theorem is indeed an entailment relation. The only
point that needs explanation is the cut rule. To simplify notations,
we take
$\ell= 3.$
We have then 3 possible cases, and we analyse only one case, where
$X,\varphi_1(z)\vdash_S Y$ and $X\vdash_S Y,\varphi_1(z)$, the other
cases being similar.
By hypothesis we have $x_1,x_2,x_3,y_1,y_2,y_3$ such that
$$\begin{array}{rclcrcl}
 x_1,\; U_0& \vda  & J_0  &\qquad &
 y_1,\; U_0& \vda  & J_0  \\
 x_2,\; U_1,\; z& \vda  & J_1 ,\; x_1  &&
 y_2,\; U_1& \vda  & J_1 , y_1 ,\; z \\
 x_3,\; U_{2}& \vda  & J_{2} ,\; x_{2}  &&
 y_3, U_{2}& \vda  & J_{2} , y_{2}  \\
 U_3& \vda  & J_3 ,\; x_3  &&
 U_3& \vda  & J_3 ,\; y_3
\end{array}$$
The two \entrels on the second line give
$$\begin{array}{rclcrcl}
 x_2,\;y_2,\; U_1,\; z& \vda  & J_1 ,\; x_1,\;y_1 &\qquad
 x_2,\;y_2,\; U_1& \vda  & J_1 ,\; x_1 ,\; y_1 ,\; z \end{array}$$
hence by cut
$$\begin{array}{rclcrcl}
 x_2,\;y_2,\; U_1& \vda  & J_1 ,\; x_1,\;y_1
\end{array}$$
\cad
$$\begin{array}{rclcrcl}
 x_2\vi y_2,\; U_1& \vda  & J_1 ,\; x_1\vu y_1
\end{array}$$
Finally, using distributivity
$$\begin{array}{rclcrcl}
 (x_1\vu y_1),\; U_0& \vda  & J_0   \\
 (x_2\vi y_2),\; U_1& \vda&J_1,\;(x_1\vu y_1)\\
 (x_3\vi y_3),\; U_{2}& \vda  & J_{2},\;(x_2\vi y_2)  \\
 U_3& \vda  & J_3, \;(x_3\vi y_3)
\end{array}$$
and hence $\varphi_0(U_0),\dots,\varphi_3(U_3)\,\vdash_S\,
\varphi_0(J_0),\dots,\varphi_3(J_3)$.\\
It is left to show that the lattice  $\Kr_\ell(T)$ defined from
$(S,\vdash_S)$ satisfied the desired universal condition.
For this it is enough to notice that the entailment relation
we have defined is clearly the least possible relation
ensuring the  $\varphi_i$ to form an increasing chain.
\end{proof}

 Notice that the morphisms $\varphi_i$ are injective: it is easily
seen that for $a,b \in T$ the relation
$\varphi_i(a)\vdash_S\varphi_i(b)$ implies
$a\vda b$, and hence that  $\varphi_i(a)= \varphi_i(b)$ implies 
$a= b.$
\subsubsection*{Comparing the two approaches}
\label{subsubsecCompar}

The analogy between the proofs of theorems
\ref{thColSimT2} and \ref{thKrJoy} is striking.
Actually these two theorems show together that an \proc
$\cC= ((J_0,U_0),\ldots,(J_\ell,U_\ell))$ collapses in $T$ \ssi the 
\proi
 $\cP= (\varphi_0(J_0),\ldots,\varphi_\ell(J_\ell);
\varphi_0(U_0),\ldots,\varphi_\ell(U_\ell))$
collapses in $\Kr_\ell(T)$.
This is not a coincidence: given the universal property that defines
$\Kr_\ell(T)$
to give a detachable \idep of $\Kr_\ell(T)$ is the same as to give an
increasing chain of detachable prime ideals of  $T$ (of length $\ell$).
One could then think that one of the two proofs is superfluous.

Classically, one could organize things as follows.
One would define first a priori the collapsus of
an \proi (resp. an \procz) as meaning that it is impossible to
refine this \proi in a \idep
(resp. to refine this \proc in an increasing chain of prime ideals).
The simultaneous collapsus theorems (theorems
\ref{lemColSimT1}~(1) and \ref{thColSimT2}~(3)) are direct
with such definitions.
Furthermore, the algebraic characterisation of the collapsus
of an \proi
$(J,U)$ (\cad $U_0\vda J_0$ for some finite subsets $U_0\subseteq U$
and $J_0\subseteq J$) are also easily established.
The description of  $\Kr_\ell(T)$ given in theorem \ref{thKrJoy}
implies then (taking into account the algebraic characterisation of
the collapsus of an \proiz) the algebraic characterisation of
the collapsus of an \procz, \cad the point $(1)$ of theorem 
\ref{thColSimT2}.

Constructively, we have defined the collapsus of an \proi (resp. of an
\procz) as
meaning the impossibility of a refinement of this \proi into a saturated
non trivial \proi{\footnote{~More precisely the double negation (\dots
impossibility \dots non trivial) has to be taken, of course,
in the form of an explicit affirmation.}}  (resp. of this \proc in a 
saturated
non trivial \procz).
To deduce the algebraic characterisation of the collapsus of an
\proc from the algebraic characterisation of the collapsus of an \proi
and of the descrition of $\Kr_\ell(T)$ (which would avoid the
``superfluous proof'') it is enough to explain how to derive
from a saturated \proc  $((I_0,F_0),\ldots,(I_\ell,F_\ell))$
an increasing chain of morphisms $(\psi_0,\ldots,\psi_\ell)$ from $T$ in 
a
\trdi with  $\psi_k^{-1}(0)= I_k$ and
$\psi_k^{-1}(1)= F_k$ ($k= 0,\ldots,\ell$). For this it is enough to 
apply
the following lemma.
\begin{lemma}
\label{lemProcEtKrl}
Let $\;\cC= ((I_0,F_0),\ldots,(I_\ell,F_\ell))$ be a saturated \proc in 
a
\trdi  $T$. Let  $T_\cC$ be the quotient \trdi of  $\Kr_\ell(T)$ by
$\varphi_0(I_0)= \cdots= \varphi_\ell(I_\ell)= 0,\;
\varphi_0(F_0)= \cdots= \varphi_\ell(F_\ell)= 1$. Let $\pi$ be the 
canonical
projection
from $\Kr_\ell(T)$ onto $T_\cC$ and $\psi_k= \pi\circ\varphi_k$.
Then $\psi_k^{-1}(0)= I_k$ and $\psi_k^{-1}(1)= F_k$ 
$\,(k= 0,\ldots,\ell)$.
\end{lemma}
\begin{proof}{Proof}
For instance $\psi_k^{-1}(0)= \left\{x\in 
T;\varphi_k(x)= _{T_\cC}0\right\}$
is equal to, by proposition \ref{propIdealFiltre}
$$ \left\{x\in 
T\;;\;\varphi_k(x),\varphi_0(F_0),\ldots,\varphi_\ell(F_\ell)
\;\vdash_{\Kr_\ell(T)}\; \varphi_0(I_0),\ldots,\varphi_\ell(I_\ell) 
\right\}
$$
\cad the set $x$ such that there exist $x_1,\ldots,x_\ell$ such that
(where $\vda$ is the \entrel of $T$)
$$\begin{array}{rcl}
x_1,\;  F_0& \vda  &  I_0  \\
x_2,\;  F_1& \vda  &  I_1 ,\; x_1  \\
    \vdots\quad & \vdots  & \quad\vdots  \\
x,x_{k+1},\;  F_k& \vda  &  I_k ,\; x_k  \\
    \vdots\quad & \vdots  & \quad\vdots  \\
x_\ell,\;  F_{\ell-1}& \vda& I_{\ell-1} ,\; x_{\ell-1}  \\
 F_\ell& \vda  &  I_\ell ,\; x_\ell  \\
\end{array}$$
Since the \proc $\cC$ is saturated one has successively
$x_1\in I_1\subseteq I_2$, $x_2\in I_2$, \dots  $x_k\in I_k$, and  
$x_\ell\in
F_\ell$, \dots,
$x_{k+1}\in F_{k+1}\subseteq F_k$, hence $x\in I_k$.
\end{proof}
\subsubsection*{Constructive definition of the Krull dimension of a 
\trdi }
\label{subsubsecDimTrdi}
Since an \proc $\cC= ((J_0,U_0),\ldots,(J_\ell,U_\ell))$ collapses in 
$T$
\ssi the \proi
 $\cP= (\varphi_0(J_0),\ldots,\varphi_\ell(J_\ell);
\varphi_0(U_0),\ldots,\varphi_\ell(U_\ell))$
collapses in $\Kr_\ell(T)$, the two variations in the definition below
of the dimension of a \trdi are equivalent.

\begin{definition}
\label{defDiTr}~
\begin{itemize}
\item [1)] An {\em \proelz} in a  \trdi $T$ is an \proc of the form
$$ ((0,x_1),(x_1,x_2),\ldots,(x_\ell,1))
$$
(with $x_i$ in $T$). 
\item [2)] A \trdi $T$ is {\em of dimension $\leq \ell-1$}
iff it satisfies one of the equivalent conditions
\begin{itemize}
\item Any \proel of length $\ell$ collapses.
\item For any sequence $x_1,\dots,x_\ell\in T$ we have
$$
\varphi_0(x_1),\dots,\varphi_{\ell-1}(x_{\ell})
\vda
\varphi_1(x_1),\dots,\varphi_{\ell}(x_{\ell})
$$
in $\Kr_{\ell}(T)$,
\end{itemize}
\end{itemize}
\end{definition}
The condition in (2) is that:
$\forall x_1,\dots,x_\ell\in T\quad \exists a_1,\dots,a_\ell\in T$ such 
that
$$\begin{array}{rclll}
a_1,\;x_1& \vda  &  0    \\
a_2,\;x_2& \vda  &  a_1,\;x_1    \\
\vdots\qquad & \vdots  & \qquad  \vdots    \\
a_\ell,\;x_\ell& \vda  &a_{\ell-1},\;x_{\ell-1}      \\
1 & \vda  &  a_\ell,\;x_\ell    
\end{array}$$

 In particular the \trdi $T$ is of dimension $\leq -1$ \ssi $1= 0$
in $T$, and it is of dimension
$\leq 0$ \ssi $T$ is a Boolean algebra (any element has a complement).

We shall not give for \trdis neither the definition of
$\dim(T)<\ell$ nor the one of $\dim(T)\geq \ell$ and we limit ourselves
to mention that
$\dim(T)>\ell$ means that $\dim(T)\leq \ell$ is impossible.
One could refine as in section \ref{subsecPsr} these definitions
when one has a primitive inequality relation in $T$.

\ss The second variant in the definition is useful for deriving
easily the simpler following characterisation.

\begin{lemma}
\label{lemDimGen}
A \trdi $T$ generated by a set $S$ is of dimension $\leq \ell-1$
\ssi for any sequence $x_1,\dots,x_\ell\in S$
$$
\varphi_0(x_1),\dots,\varphi_{\ell-1}(x_{\ell})
\vda
\varphi_1(x_1),\dots,\varphi_{\ell}(x_{\ell})
$$
in $\Kr_{\ell}(T)$.
\end{lemma}
Indeed using distributivity, one can deduce
$$ a\vu a',A\vda b\vu b',B$$
from $a,A\vda b,B$ and $a',A\vda b',B$. Furthermore any element of $T$
is an inf of sups of elements of $S$.

 Notice the analogy bewteen the formulation of this condition and
the definition of pseudo regular sequence  \ref{defproch}.

\subsubsection*{Connections with Joyal's definition}
\label{subsubsecJoyal2}

 Let $T$ be a distributive lattice, Joyal \cite{esp} gives the following
definition of $\dim(T)\leq \ell-1$. Let
$\varphi^\ell_i:T\rightarrow \Kr_\ell(T)$ be the $\ell+1$ universal 
morphisms.
By universality of $\Kr_{\ell+1}(T)$, we have
$\ell+1$ morphisms $\sigma_i:\Kr_{\ell+1}(T)\rightarrow
\Kr_\ell(T)$
such that $\sigma_i\circ \varphi^{\ell+1}_j =  \varphi^\ell_j$ if 
$j\leq i$
and $\sigma_i\circ \varphi^{\ell+1}_j =  \varphi^\ell_{j-1}$ if $j>i$.
Joyal defines then $\dim(T)\leq \ell$ to mean that
$(\sigma_0,\dots,\sigma_\ell):\Kr_{\ell+1}(T)\rightarrow
\Kr_\ell(T)^{\ell+1}$
is injective. This definition can be motivated by proposition
\ref{propRep2}: the elements in the image of
de $Sp(\sigma_i)$ are the chains of prime ideals
$(\alpha_0,\dots,\alpha_{\ell})$
with $\alpha_i= \alpha_{i+1}$, and $Sp(\sigma_0,\dots,\sigma_\ell)$
is surjective \ssi for any chain $(\alpha_0,\dots,\alpha_{\ell})$
there exists $i<\ell$ such that $\alpha_i= \alpha_{i+1}$. This means 
exactly
that there is no non trivial chain of prime ideals of length $l+1$.
Using \tcgz, one can then see the equivalence
with definition  \ref{defDiTr}.
One could check directly this equivalence using a constructive
metalanguage, but
for lack of space, we shall not present here this argument.
Similarly, it would be possible to establish the equivalence of our
definition with the one of Espan\~ol \cite{esp} (here also, this 
connection
is clear via \tcgz).

\section{Zariski and Krull lattice associated to a commutative ring}
\label{secZariKrull}
\subsubsection*{Zariski lattice}
\label{subsubsecZar}
Given a commutative ring $R$ the {\em Zariski lattice}  $\Zar(R)$ has 
for
elements the radical ideals (the order relation being inclusion).
It is well defined as a lattice.
Indeed $\sqrt{I_1}= \sqrt{J_1}$ and $\sqrt{I_2}= \sqrt{J_2}$
imply
$\sqrt{I_1I_2}=  \sqrt{J_1J_2} $
(which defines $\sqrt{I_1}\vi\sqrt{I_2}$) and
$\sqrt{I_1+I_2}= \sqrt{J_1+J_2}$
(which defines $\sqrt{I_1}\vu\sqrt{I_2}$). The Zariski lattice of $R$ is
always distributive, but may not be decidable.
Nevertheless an inclusion  $\sqrt{I_1}\subseteq \sqrt{I_2}$ can always
be certified in a finite way if the ring $R$ is discrete.
This lattice contains all the informations necessary for a constructive
development of the abstract theory of the Zariski spectrum.\\
We shall write $\wi{a}$ for $\sqrt{\gen{a}}$. Given a subset $S$ of $A$
we write $\wi{S}$ the subset of $\Zar(R)$ the  elements of which are
$\wi{s}$ for
$s\in S$. We have
$\wi{a_1}\vu\cdots\vu\wi{a_m}= \sqrt{\gen{a_1,\ldots,a_m}}$ and
$\wi{a_1}\vi\cdots\vi\wi{a_m}= \wi{a_1\cdots a_m}$ \\
Let $U$ and  $J$ be two finite subsets of $R$, we have
$$ \wi{U}\,\vdash_{\Zar(R)}\wi{J}
\quad\Longleftrightarrow \quad
\prod_{u\in U} u  \in \sqrt{\gen{J}}
\quad\Longleftrightarrow \quad
\cM(U)\cap \gen{I}\neq \emptyset
$$
\cad
$$ (J,U){\rm \;collapses \;in  \;  } R
\quad\Longleftrightarrow \quad
(\wi{J},\wi{U}){\rm \;collapses \;in  \;  } \Zar(R)
$$
This describes completely the lattice $\Zar(R)$. More precisely we have:
\begin{proposition}
\label{propZar} The lattice $\Zar(R)$ of a commutative ring $R$ is
(up to isomorphsim) the lattice generated by  $(R,\vda)$  where $\vda$
is the least entailment relation over $R$ such that
$$\begin{array}{rclcrclcrcl}
  0_A   & \vda  &      &\qquad  & x,\; y & \vda  &  x y   \\
        & \vda  &  1_A &\qquad &  xy    & \vda  &   x   &\qquad &
  x+y   & \vda  & x,\; y  \\
\end{array}$$
\end{proposition}
\begin{proof}{Proof}
It is clear that the relation $U\vda J$ defined by ``$\cM(U)$
meets $\gen{J}$'' satisfies these axioms. It is also clear that
the entailment relation generated by these axioms contains this
relation. Let us show that this relation is an  \entrelz.
Only the cut rule is not obvious.
Assume that $\cM(U,a)$ meets
$\gen{J}$ and that $\cM(U)$ meets $\gen{J,a}$. There exist then
$m_1,m_2\in \cM(U)$ and $k,x$ such that
$a^k m_1 \in \gen{J},~m_2+ax\in \gen{J}$.
Eliminating $a$ this implies that $\cM(U)$ intersects $\gen{J}.$
\end{proof}
We have $\wi{a}= \wi{b}$ \ssi $a$ divides a power of $b$ and
$b$ divides a power of $a$.
\begin{proposition}
\label{propZar2}
In a commutative ring $R$ to give an ideal of the lattice
$\Zar(R)$ is the same as to give a radical ideal of $R$.
If $I$ is a radical ideal of $R$ one associates the ideal
$${\cI} =  \{ J \in \Zar(R)~|~ J \subseteq I\}$$
of $\Zar(R)$. Conversely if $\cal{I}$ is an ideal of
$\Zar(R)$ one can associate the ideal
$$I= \bigcup _{J\in\cal{I}}J= \{ x\in A~|~\wi{x}\in \cal{I}\},$$
which is a radical ideal of $R.$
In this bijection the prime ideals of the ring correspond to
the prime ideals of the Zariski lattice.
\end{proposition}
\begin{proof}{Proof}
We only prove the last assertion.
If $I$ is a prime ideal of $R$, if $J,J'\in \Zar(R)$ and
$J\vi J'\in \cal{I}$, let $a_1,\dots,a_n\in R$ be some ``generators'' of
$J$ (\cad $J= \sqrt{\gen{a_1,\dots,a_n}}$) and let $b_1,\dots,b_m\in A$ 
be some
generators of $J'.$ We have then
$a_ib_j\in I$ and hence $a_i\in I$ or $b_j\in I$ for all $i,j.$ It 
follows
from this (constructively) that we have $a_i\in I$ for all $i$ or
$b_j\in I$ for all $j$. Hence $J\in \cal{I}$ or $J'\in \cal{I}$ and
$\cal{I}$ is a prime ideal of $\Zar(R).$
\\
Conversely if $\cal{I}$ is a prime ideal of $\Zar(R)$
and if we have $\wi{xy}\in \cal{I}$ then
$\wi{x}\vi \wi{y}\in \cal{I}$ and hence $\wi{x}\in \cal{I}$
or $\wi{y}\in \cal{I}$. This shows that $\{ x\in A~|~\wi{x}\in 
\cal{I}\}$
is a prime ideal of $R$.
\end{proof}
\begin{definition}
\label{defKruA}
We define $\Kru_\ell(R):= \Kr_\ell(\Zar(R))$. This is called the
{\em Krull lattice of order $\ell$} of the ring $R$.
\end{definition}

\begin{theorem}
\label{corCollaps} Let $\cC= ((J_0,U_0),\ldots,(J_\ell,U_\ell))$ be
an \proc in a commutative ring $R$.
It collapses \ssi the \proc
$((\wi{J_0},\wi{U_0}),\ldots,
(\wi{J_\ell},\wi{U_\ell}))$  collapses in
$\Zar(R)$. For instance if $\cC$ is finite, \propeq
\begin{enumerate}
\item  there exist $j_i\in \gen{J_i}$, $u_i\in\cM(U_i)$,
$(i= 0,\ldots,\ell)$,
such that
$$u_0\cdot(u_1\cdot(\cdots(u_\ell+j_\ell)+\cdots)+j_1)+j_0= 0
$$
\item  there exist $x_1,\ldots,x_\ell\in \Zar(R)$ such that
in $\Zar(R)$:
$$\begin{array}{rcl}
 x_1,\; \widetilde{U_0}& \vda  &  \widetilde{J_0}
\\
 x_2,\; \widetilde{U_1}& \vda  &  \widetilde{J_1} ,\;  x_1
\\
\vdots\qquad & \vdots  & \qquad\vdots
\\
x_\ell,\; \widetilde{U_{\ell-1}}&\vda& \widetilde{J_{\ell-1}},\;x_{\ell-
1}
\\
\widetilde{U_\ell}& \vda & \widetilde{J_\ell} ,\; x_\ell
\end{array}$$
\item  same thing but with $x_1,\ldots,x_\ell\in \wi{A}$
\end{enumerate}
\end{theorem}
\begin{proof}{Proof}
It is clear that $1$ entails $3$: simply take
$$v_\ell =  u_\ell+j_\ell,~v_{\ell-1} =  v_\ell u_{\ell-1} +
j_{\ell-1},\dots,~v_0 =  v_1u_0 + j_0
~~~~{\rm  and}~~~ x_i=  \wi{v_i}
$$
and that $3$ entails $2$.
The fact that $1$ follows from $2$ can be seen by reformulating
$2$ in the following way. We consider the
\proc $\cC_1= ((K_0,V_0),\ldots,(K_\ell,V_\ell))$ obtained by 
saturating
the \proc $\cC$.
We define the $\ell+1$ radical ideals
$I_0,\dots,I_{\ell}$ of $R$
\begin{itemize}
\item $I_0 =  \{x\in A~|~\cM(x,U_0)\cap \gen{J_0} \neq \emptyset \}$
\item $I_1 =  \{x\in A~|~\cM(x,U_1)\cap (\gen{J_1} + I_0) \neq 
\emptyset\}$
\item ~~$\vdots$
\item $I_{\ell-1} =  \{x\in A~|~\cM(x,U_{\ell-1})\cap (\gen{J_{\ell-1}} 
+
I_{\ell-2}) \neq \emptyset \}$
\item $I_\ell= \gen{J_\ell}+I_{\ell-1}$
\end{itemize}
It is clear that $I_i\subseteq K_i$ ($i= 0,\ldots,\ell$).
In the correspondance given in \ref{propZar2} these ideals
correspond to the following ideals of $\Zar(R)$
  \begin{itemize}
\item ${\cI}_0 = 
  \{u\in \Zar(R)~|~ u,\;  \widetilde{U_0} \vda  \widetilde{J_0}\}$
\item ${\cI}_1 = 
  \{u\in \Zar(R)~|~ (\exists v\in {\cI}_0)~
                    u,\;  \widetilde{U_1} \vda \widetilde{J_1},\;
v\}$
\item ~~$\vdots$
\item ${\cI}_{\ell-1} = 
  \{u\in \Zar(R)~|~ (\exists v\in {\cI}_0)~
                     u,\;  \widetilde{U_{\ell-1}} \vda
\widetilde{J_{\ell-1}},\; v\}$
\end{itemize}
The condition $2$ becomes then
$\;\widetilde{U_{\ell}}\vda  \widetilde{J_{\ell}},\; v\;$ for some
$\;v\in {\cI}_{\ell-1}$.
This means that $\cM(U_{\ell})$ intersects
$I_{\ell}$, or $I_{\ell}\subseteq K_{\ell}$.
Hence $\cC_1$ collapses, and hence $\cC$ collapses.

Let us give another direct proof that $(2)$ implies $(3)$.
We rewrite the \entrels of (2) as follows.
Each $\wi{U_i}$ can be replaced by a $\wi{u_i}$ with $u_i\in R$,
each $\wi{J_i}$ can be replaced by a radical of \itf $I_i$  of
$R$, and we write $L_i$ instead of $x_i$ to indicate that this  is
a radical of a  \itfz. We get:
$$\begin{array}{rcl}
 L_1,\; \widetilde{u_0}& \vda  &  I_0  \\
 L_2,\; \widetilde{u_1}& \vda  &  I_1 ,\;  L_1  \\
 L_3,\; \widetilde{u_2}& \vda  &  I_2 ,\; L_2  \\
\widetilde{u_3}& \vda & I_3 ,\; L_3
\end{array}$$
The last line means that $\cM(u_3)$ intersects $I_3+ L_3$
and hence  $I_3+ \gen{y_3}$ for some element $y_3$ of
$L_3$. Hence we have
$ \wi{u_3} \vda  I_3 ,\; \wi{y_3}$.
Since $\wi{y_3}\leq L_3$ in $\Zar(R)$ we have
 $ \;\wi{y_3},\; \wi{u_2} \vda  I_2 ,\; L_2$.
We have then replaced  $L_3$ by $\wi{y_3}$.
Reasoning as previously one sees that one can replace as
well $L_2$ by a suitavle $\wi{y_2}$, and then
$L_1$ by a suitable $\wi{y_1}$. One gets then (3).
\end{proof}

\begin{corollary}
\label{corCompDim}
The Krull dimension of a commutative ring $R$ is $\leq \ell$ \ssi
the Krull dimension of its Zariski lattice $\Zar(R)$ is $\leq \ell$.
\end{corollary}

\begin{proof}{Proof}
By the previous theorem and lemma \ref{lemDimGen}.
\end{proof}

 This would be a natural place to relate decidability
properties of $R$ and of  $\Kr_n(\Zar(R))$. For instance,
it can be shown that if $R$ is coherent, noetherian and
strongly discrete then each of the $\Kr_n(\Zar(R))$ is decidable.
Due to lack of space, we shall not present these results here.

\section{Going Up and Going Down}
\label{secGUGD}
\subsection{Relative Krull dimension}
\label{subsecKrullRel}
\subsubsection*{General remarks about relative Krull dimension}

We shall develop here a constructive counterpart of the notion of
increasing chain of prime ideals which all lie over the same
prime ideal of a given subring.
This paragraph can apply as well to the case of an arbitary \trdi
(here it is $\Zar(R)$) with evident modifications. There is no
real computations going on, just some simple combinatorics.
\begin{definition}
\label{defColRel}
Let $R\subseteq S$ be two commutative rings and
$\cC= ((J_0,U_0),\ldots,(J_\ell,U_\ell))$ an \proc in $S$.
\begin{itemize}
\item {\em The \proc $\cC$ collapses above $R$} \ssi there exist
$a_1,\ldots,a_k\in R$ such that for all couple of complementary subsets
$(H,H')$ of $\{1,\ldots,k\}$, the \proc
$$ (\{(a_h)_{h\in H}\}\cup J_0,U_0),(J_1,U_1)\ldots,(J_\ell,U_\ell\cup
\{(a_h)_{h\in H'}\})
$$
collapses.
\item {\em The (relative) Krull dimension of the extension $S/R$
is $\leq  \ell-1$} \ssi any \proel  
$((0,x_1),(x_1,x_2),\ldots,(x_\ell,1))$
collapses above $R$.
\item {\em The (relative) Krull dimension of the extension $S/R$
is $\geq \ell$} \ssi there exist $x_0,\ldots,x_\ell$ in $S$ such that 
the
\proel $((0,x_1),(x_1,x_2),\ldots,(x_\ell,1))$
does not collapse{\rm (}{\footnote{~More precisely, constructively,
we have to say: for any $k$ and any $a_1,\ldots,a_k\in R$
there exist a pair of complementary subsets $(H,H')$ of 
$\{1,\ldots,k\}$,
such that the \proc
$$ (\{(a_h)_{h\in H}\};x_1),(x_1,x_2),\ldots,(x_\ell;\{(a_h)_{h\in 
H'}\})
$$
``does not collapse'' with the meaning of the 
inequality relation defined over $R$
(cf. the explanation in the beginning of section \ref{subsecPsr} page
\pageref{nbpIneq}).}}{\rm )} above $R$.
\item {\em The (relative) Krull dimension of the extension $S/R$
is $< \ell$} \ssi it is impossible that it is $\geq \ell$.
\item {\em The (relative) Krull dimension of the extension $S/R$
is $> \ell$} \ssi it is impossible that it is $\leq \ell$.
\end{itemize}
\end{definition}
One can consider a more general case of a ring extension: a map
$R\rightarrow S$ non necessarily injective.
It is possible to adapt the previous definition by replacing $R$
by its image in $S$.

One has a relative simultaneous collapse theorem.
\begin{theorem}
\label{thColSimRel} {\em (Relative simultaneous collapse for the 
\procsz)}
Let $R\subseteq S$ be two commutative rings and $\cC$ an \prolo $\ell$ 
in
$S$.
\begin{itemize}
\item [$(1)$] Take $x\in S$ and $i\in\left\{0,\ldots ,\ell\right\}$.
Suppose that the \procs
$\cC\,\& \left\{x\in \cC^{(i)} \right\} $
and
$\cC\,\& \left\{x\notin \cC^{(i)} \right\} $
both collapse above $R$, then so does $\cC$.
\item [$(2)$] Take $x\in R$. Suppose that the \procs
$\cC\,\& \left\{x\in \cC^{(0)} \right\} $
and
$\cC\,\& \left\{x\notin \cC^{(\ell)} \right\} $
both collapse above $R$, then so does $\cC$.
\end{itemize}
\end{theorem}
This is an easy consequence of the (non relative) theorem
\ref{ThColSimKrA}, that is left to the reader.
  From this, we deduce (classically)
a characterisation of the  \procs which collapse relatively.

\begin{theorem}
\label{th.nstformelRel} {\em (Formal \nst for the chains of prime ideals
in a ring extension)} \Tcgi Let $R\subseteq S$ be commutative rings
and $\cC= ((J_0,U_0),\ldots,(J_\ell,U_\ell))$ an
\proc in $S$. \Propeq
\begin{itemize}
\item [$(a)$] There exists a detachable prime ideal $P$ of $R$ and
$\ell+1$ detachable prime ideals
$P_0\subseteq \cdots\subseteq P_\ell$ of $S$ such taht
$J_i\subseteq P_i$, $U_i\cap P_i= \emptyset $ and $P_i\cap A= P$
$(i= 0,\ldots,\ell)$.
\item [$(b)$] The \proc $\cC$ does not collapse above $S$.
\end{itemize}
\end{theorem}
\begin{proof}{Proof}
We have clearly $(a)\Rightarrow (b)$. For proving $(b)\Rightarrow (a)$ 
we
do the (easier) proof which relies on \pte and Zorn's lemma.
We consider a maximal \proc
$\cC_1= ((P_0,S_0),\ldots,(P_\ell,S_\ell))$
(for the extension relation) among all the \procs which refines $\cC$
and that do not collapse above $R$. Given the relative simultaneous 
collapse
theorem, the same proof that in theorem \ref{th.nstformel} shows that
it is an increasing  chain of prime ideals (with their complements).
It is left to show that all $P_i\cap A$  are equals, which is equivalent 
to
$S_0\cap P_\ell\cap A= \emptyset$.
If this was not so we would have $x\in S_0\cap P_\ell\cap A$.
Then $((P_0\cup \{x \},S_0),\ldots,(P_\ell,S_\ell))$ and
$((P_0,S_0),\ldots,(P_\ell,S_\ell\cup \{x \}))$ collapses (absolutely) 
and
hence
$\cC_1$ collapses above $A$  (with the finite subset $\{x\}$). This is 
absurd.
\end{proof}

Constructively we have the following result. We omit the proof for
reason of space.
\begin{theorem}
\label{thDim1}
Let $R\subseteq S$ be commutative rings.
\begin{itemize}
\item [$(1)$] Suppose that the Krull dimension of $R$ is $\leq m$
and that the relative Krull dimension of the extension $S/R$ is $\leq 
n$,
then the Krull dimension of $B$ is $\leq (m+1)(n+1)-1$.
\item [$(2)$] Suppose that $R$ and $S$ have an inequality $\neq0$
defined as the negation of $= 0$.
Suppose that the Krull dimension of the extension $S/R$ is
 $\leq n$ and that the collapse of \proels in $R$ is decidable. Given
a pseudo regular sequence of length $(m+1)(n+1)$  in $B$,
on can build a pseudo regular sequence of length $m+1$ in $R$.
\end{itemize}
\end{theorem}

\subsubsection*{Case of integral extensions }

In the following proposition $(1)$ is the constructive version
of the ``incompatibility theorem'' (theorem 13.33 in the book
of Sharp cite{Sha}).

\begin{proposition}
\label{propDim2}  Let $R\subseteq S$ be commutative rings.
\begin{itemize}
\item [$(1)$] If $S$ is integral over $R$ the relative Krull dimension
of the extension $S/R$ is $0$.
\item [$(2)$] More generally we have the same result if
any element of $S$ is a zero of a polynomial in $R[X]$ which has
a coefficient equal to $1$. For instance if $R$ is
a Pr\"ufer domain, this applies to any overring of $R$
in its quotient field.
\item [$(3)$] In particular, using theorem \ref{thDim1}
if  $\dim(R)\leq n$ then $\dim(S)\leq n$.
\end{itemize}
\end{proposition}
\begin{proof}{Proof}
We show $(2)$. We have to show that for any $x\in S$ the \proc
$((0,x),(x,1))$ collapses above $R$.
The finite list in $R$ is the one given by the coefficients of
the polynomial of which $x$ is a zero. Suppose that
$x^k= \sum_{i\neq k,i\leq r}a_ix^i$.
Let  $\,G,\; G'\,$ be two complementary subsets of
$\left\{a_i;i\neq k\right\}$.
The collapsus of $((G,x),(x,G'))$ is of the form
$\,x^m(g'+bx)= g\,$ with $g\in \gen{G}_S,\; g'\in \cM(G'),\; b\in S$.
Actually we take $g\in G[x]$ and $b\in R[x]$.
If $G'$ is empty we take $m= k,\; g'= 1$.
Otherwise let $h$ be the smallest index $\ell$ such that $a_\ell\in G'$.
All $a_j$ with $j<h$ are in $G$. If $h<k$ we take $m= h,\; g'= a_h$.
If $h>k$ we take $m= k,\; g'= 1$. \\
NB: notice that the disjunction has only $r$ cases and not $2^r$:
\begin{itemize}
\item $ a_0\in G'$, or
\item $  a_0\in G,a_1\in G'$, or
\item $  a_0,a_1\in G,a_2\in G'$, or
\item $~~~~~~~\vdots$
\end{itemize}
\end{proof}

\subsubsection*{Relative Krull dimension with polynomial rings}
We give a constructive version of the classical theorem on the
relative Krull dimension of an extension
$A[x_1,\ldots, x_n]/A$.

We shall need the following elementary lemma from linear algebra.

\begin{lemma}
\label{lemALE} Let $V_1,\ldots,V_{n+1}$ be vectors in $R^{n}$.
\begin{itemize}
\item  If $R$ is a discrete field, there exists an index
$k\in \left\{1,\ldots,n+1\right\}$ such that
$V_k$ is a \coli of the following vectors (if $k= n+1$ this means
$V_{n+1}= 0$).
\item If $R$ is a commutative ring, write $V$ the matrix the columns of
which are the  $V_i$. Let $\mu_1,\ldots,\mu_\ell$
(with $\ell= 2^n-1$) be the list of all minors of $V$ extracted on the
$n$ or $n-1$ or $\ldots$ or $1$  last columns, and ranked by decreasing 
size.
Take $\mu_{\ell+1}= 1$
(the corresponding minor for the empty extracted matrix).
For each $k\in\left\{1,\ldots,\ell+1\right\}$ we take
$I_k= \gen{(\mu_i)_{i< k}}$ and $S_k= \cS(I_k;\mu_k)$.
If the minor $\mu_k$ is of order $j$, the vector $V_{n+1-j}$
is, in the ring $(R/I_k)_{S_k}$, equal to a  \coli of the following 
vectors.
\end{itemize}
\end{lemma}
\begin{proof}{Proof}
For the second point, one uses Cramer formulas.
\end{proof}

\begin{proposition}
\label{propKrRelPol} Let $S= R[X_1,\ldots, X_n]$ be a poynomial ring.
The relative Krull dimension of the extension $S/R$ is equal to $n$.
Hence if the Krull dimension of $R$ if $\le r$ the one of $S$ is  $\le 
r+n+rn$.
Furthermore if the Krull dimension of $S$ is $\le r+n$ the one of  $R$ 
is
$\le r$.
\end{proposition}

\begin{proof}{Proof}
The last assertion follows from the fact that if the sequence
 $(a_1,\ldots ,a_r,X_1,\ldots,X_n)$ is pseudo singular in $S$
then the sequence $(a_1,\ldots ,a_r)$ is pseudo singular in $R$:
we have indeed, considering only the case $m= n= 2$, an equality
in $S$ of the form:
$$ a_1^{m_1}a_2^{m_2}X_1^{p_1}X_2^{p_2}+
a_1^{m_1}a_2^{m_2}X_1^{p_1}X_2^{p_2+1}R_4+
a_1^{m_1}a_2^{m_2}X_1^{p_1+1}R_3+
a_1^{m_1}a_2^{m_2+1}R_2+
a_1^{m_1+1}R_1 =  0
$$
Looking the coefficient  of $X_1^{p_1}X_2^{p_2}$ in the polynomial
of the left hand-side we get
$$ a_1^{m_1}a_2^{m_2}+
a_1^{m_1}a_2^{m_2+1}r_2+
a_1^{m_1+1}r_1 =  0
$$
which gives the collapsus of $\ov{(a_1,a_2)}$ in $R$.\\
The second assertion follows from the first (cf. theorem
\ref{thKDP}(1)).\\
The proof of the first assertion in classical mathematics
relies directly on considering the case of fields.
We give a constructive proof which follows the same pattern,
and consider the case of (discrete) fields.
We analyse the proof of proposition \ref{propKrDimetDegTr}
and we substitute everywhere the field $K$ by a ring $R$, which
will allow us to use the definition of collapsus above $R$.
Take $(y_1,\ldots,y_{n+1})$ in $R[X_1,\ldots, X_n]$.
We can write as a proof in linear algebra a proof that
the $y_i$ are algebraically dependent over $R$, assuming first
that $R$ is a discrete field.
For instance if the $y_i$ are polynomials of degre $\le d$ the
polynomials
$y_1^{m_1}\cdots y_{n+1}^{m_{n+1}}$ with  $\sum_i m_i\le m$ are
in the vector space of polynomials of degre $\le dm$ which is
of dimension
$\le {dm+n\choose n}$, and there are  ${m+n+1 \choose n+1}$ of them.
For an explicit value of $m$ we have
${m+n+1 \choose n+1}>{dm+n\choose n}$ (since one term is a polynomial
of degre $n+1$ and the other a polynomial of degre $n$).
We fix $m$ to this value.
We order the corresponding ``vectors''
 $y_1^{m_1}\cdots y_{n+1}^{m_{n+1}}$
(such that  $\sum_i m_i\le m$) along the lexicographic order for
$(m_1,\ldots,m_{n+1})$. We can limit ourselves to consider ${dm+n\choose 
n}+1$
vectors.
Using lemma \ref{lemALE}, we get in each rings $(R/I_k)_{S_k}$
a vector
$y_1^{m_1}\cdots y_{n+1}^{m_{n+1}}$ which is a linear combination of
the following vectors. This gives, like in the proof of
proposition \ref{propKrDimetDegTr} a collapsus, but this time
we have to add at the beginning and at the end of the \proel
$\overline{(y_1,\ldots,y_{n+1})}$ the ``additional hypothesis'': the
\proc
$$
((\mu_i)_{i<k},y_1;y_2),(y_2;y_3)\ldots,
      (y_{n-1},(y_{n}),(y_{n};y_{n+1},\mu_k)
$$
collapses (for each $k$). Indeed, for showing the collapsus of an
\proc one can always quotient by the first of the ideals (or
by a smaller ideal) and localised along the last of the \mos (or along
a smaller \mo).\\
Lastly, all these collapsus give us the collapsus of
$\overline{(y_1,\ldots,y_{n+1})}$ above $R$
using the finite set of the  $\mu_i$.
\end{proof}

\subsection{Going Up}
\label{subsecGU}
If $R$ is a subring of $S$, an \proc of $R$ which collapses in $R$
collapses in $S$ and the trace on $R$ of a saturated \proi of $S$
is a saturated \proi of $R$. Among other things, we shall establish
in this section that for integral extensions we have also the converse
of these assertions.

\begin{lemma}
\label{lemGU1} Let $R\subseteq S$ be two commutative rings where
$S$ is integral over $R$. Let $I$ be an ideal of $R$ and take $x\in R.$
Then
$$ x\in\sqrt{I}\quad \Longleftrightarrow \quad x\in\sqrt{IB}
$$
\end{lemma}
\begin{proof}{Proof}
Suppose  $ x\in\sqrt{IB}$, \cad $ x^n= \sum j_ib_i$, with
$j_i\in I,\;b_i\in B$. The $b_i$ and 1 generate a sub $R$-module 
faithful
and \tf $M$ of $S$ and
$x^n$ can be written as a linear combinaison  with coefficients in $I$
over a system of generators of $M$. The  \polcar of the matrix of
the multiplication by $x^n$ (expressed using this system of generators)
have then all its coefficients (except maybe the leading one) in $I$.
\end{proof}
\begin{definition}
\label{defTraceproc} Let $R\subseteq S$ two commutative rings.
Take $\cP= (J,V)$ an \proi of $S$ and $\cC= (\cP_1,\ldots \cP_n)$ an 
\proc de
$S$.
We say that {\em $(J\cap R,V\cap R)$ is the trace of $\cP$ on $R$}.
We write $\cP|_R$ this \proi of $R$.  We say that
{\em $(\cP_1|_R,\ldots \cP_n|_R)$ is the trace of $\cC$ one $R$}.
We write $\cC|_R$ this \proc of  $R$.
\end{definition}
It is clear that the trace of a complete (resp. saturated) \proc
is complete (resp. saturated).

\begin{corollary}
\label{corLyO}
{\rm (Lying over)} Let $R\subseteq S$ be commutative rings with
$S$ integral over $R$.
\begin{itemize}
\item  Let $\cP$ be an \proi in $R$.
\begin{itemize}
\item  [$(1)$] If $\cP$ collapses in  $S$, it collapses in $R$.
\item  [$(2)$] If $\cQ$ is the saturation of $\cP$ in  $S$, then
$\cQ|_R$ is the saturation of $\cP$ in $R$.
\end{itemize}
%
\item \Tcgi Any prime ideal of $R$ is the trace on $A$ of a prime ideal
of $S$.
\end{itemize}
\end{corollary}
\begin{proof}{Proof}
We show (1), the other points are then easy consequences.
If $\cP= (I,U)$ collapses in  $S$ an element of $\cM(U)$ is in the 
radical
of  $\gen{I}_S= \gen{I}_RS$ and hence, by the previous lemma, is in
the radical of $\gen{I}_R$.
\end{proof}

\begin{theorem}
\label{thGU} {\em (Going Up)}
 Let $R\subseteq S$ be commutative rings with
$S$ integral over $R$.
Let $\cC_1$ be a saturated  \proc of $S$ and $\cC_2$
an \proc of $R$.
\begin{itemize}
\item [$(1)$] The \proc
$\cC= \cC_1\bullet \cC_2$
collapses in $S$ \ssi the \proc $\cC_1|_R\bullet \cC_2$ collapses in 
$R$.
\item [$(2)$] Let $\cC'$ be the saturation of $\cC$ in $S$.
The trace of $\cC'$ on $R$ is the saturation of
$\cC_1|_R\bullet \cC_2$ in $R$.
\end{itemize}
In particular any \proc of $R$ that collapses in $S$ collapses
in $R$, and the trace on $R$ of the saturation in $S$ of
an \proc of $R$ is equal to its saturation in $R$.
\end{theorem}
\begin{proof}{Proof}
Take
$\cC_1= ((J_1,V_1),\ldots,(J_\ell,V_\ell))$ (in $S$),
$\cC_1|_R= ((I_1,U_1),\ldots,(I_\ell,U_\ell))$ its trace on $R$ and
$\cC_2= ((I_{\ell+1},U_{\ell+1}),\ldots,(I_{\ell+r},U_{\ell+r}))$.
Write $(1)_{\ell,r}$ and $(2)_{\ell,r}$ the assertions for given $\ell$ 
and
$r$.
These two statements are in fact directly equivalent, given
the characterisation of the saturation of an \proc in term
of collapsus. Notice also that $(1)_{0,1}$ and
$(2)_{0,1}$ state the Lying over.\\
We show $(2)_{0,r}\Rightarrow (1)_{\ell,r}$. The \proc
$\cC_1|_R$ is saturated in  $R$.
Consider the quotient ring
$R'=  R/I_\ell\subseteq S'= S/J_\ell$. The ring $S'$ is still integral 
over $R'$.
$(2)_{0,r}$ applied to these quotients gives $(1)_{\ell,r}$
using fact \ref{factSatur}~(1).\\
It is enough now to show $(2)_{1,r}\Rightarrow (1)_{0,r+1}$.
Let $(\cP_1,\ldots,\cP_{r+1})$ be an \proc in $R$ which collapses in 
$S$.
Let $\cQ_1$ be the saturation of $\cP_1$ in $S$. The trace of $\cQ_1$ 
on $R$
is $\cP_1$ by the Lying over. We can then apply $(2)_{1,r}$ with
$\cC_1= \cQ_1$ and  $\cC_2= (\cP_2,\ldots,\cP_{r+1})$.
\end{proof}
\begin{corollary}
\label{corGU1}
{\rm (Going Up, classical version)} \Tcgi
 Let $R\subseteq S$ be commutative rings with
$S$ integral over $R$.
Let $Q_1\subseteq\cdots\subseteq Q_\ell$ be prime ideals in $S$,
$P_i= Q_i\cap R$  $(i= 1,\ldots,\ell)$, and
$P_{\ell+1}\subseteq\cdots\subseteq P_{\ell+r}$
prime ideals in  $R$ with $P_{\ell}\subseteq P_{\ell+1}$.
There exist $Q_{\ell+1},\ldots,Q_{\ell+r}$, prime ideals in $S$ which
satisfy $Q_\ell\subseteq Q_{\ell+1}\subseteq\cdots\subseteq Q_{\ell+r}$
and $Q_{\ell+j}\cap A= P_{\ell+j}$ for $j= 1,\ldots,r$.
\end{corollary}
\begin{proof}{Proof}
We consider the \procs $\cC_1$ (in $S$) and
$\cC_2$ (in $R$) associate to the given chains.
by hypothesis,  $\cC_1|_R\bullet \cC_2$ does not collapse in $R$
and hence, by constructive Going up,  $\cC= \cC_1\bullet \cC_2$ does
not collapse in $S$. We can then consider a chain of prime ideals
of $S$ that refines the \proc $\cC$
(theorem \ref{th.nstformel}). Since $\cC_1$ is frozen it does not
change in this process (otherwise it would collapse).
The last part of the chain
$Q_{\ell+1},\ldots,Q_{\ell+r}$,
has its trace on $R$ frozen, and hence it has to be equal to
$P_{\ell+1},\ldots,P_{\ell+r}$ (otherwise it would collapse).
\end{proof}
Notice that it seems difficult to show directly, in classical 
mathematics,
theorem \ref{thGU} from corollary \ref{corGU1},
even using theorem \ref{th.nstformel} which connects the
\procs  and the chains of prime ideals.
\begin{corollary}
\label{corGU3}
{\rm (Krull dimension of an integral extension)}
 Let $R\subseteq S$ be commutative rings with
$S$ integral over $R$.
\begin{itemize}
\item [$(1)$] The Krull dimension of $R$ is $\leq n$ \ssi the Krull
dimension of $S$ is $\leq n$.
\item [$(2)$] A pseudo regular sequence of $R$ is pseudo regular in $S$.
\item [$(3)$] If the collapse of \proels in $R$ is decidable,
from a psuedo regular sequence in $S$ one can build a pseudo regular
sequence of same length in $R$.
\end{itemize}
NB: for the  points $(2)$ and $3$ the rings are supposed to be equipped
with the apartness relation $\lnot(x= 0)$.
\end{corollary}
\begin{proof}{Proof}
(1) Proposition \ref{propDim2}~(3) gives
$\dim(R)\le n\Rightarrow \dim(S)\le n$, and the last part of theorem
\ref{thGU} gives the converse.\\
(2) follows by contrapposition from the last part of theorem 
\ref{thGU}.\\
(3) Since the relative Krull of $S$ on $R$ is $0$, we can apply theorem
 \ref{thDim1}~(2).
\end{proof}
A corollary of the previous result and of theorem \ref{thKDP} is
the following theorem, which says that the Krull dimension of
a finitely presented algebra  over a discrete field is
the one given by Noether's normalisation lemma.

\begin{theorem}
\label{thKDP2} Let $K$ be a discrete field,  $I$ a \tf ideal of 
the
ring
$K[X_1,\ldots,X_\ell]$ and $A$ the quotient algebra.
Noether's normalisation lemma applied to the ideal $I$
gives an integer $r$ and elements $y_1,~ \ldots,~ y_r$ of $A$ that
are algebraically independent over $K$ and such that $A$ is
a \tf module over $K[y_1,\ldots,y_r]$.
The Krull dimension of $A$ is equal to $r$.
\end{theorem}

\subsection{Going Down}
\label{subsecGD}
This section is copied over the section on Going Down given
in Sharp's book. Since one cannot in general compute the
minimal polynomial of an element algebraic over a field,
the first lemma is a little less simple than the corresponding
lemma in classical mathematics and the proof of the Going Down
actually explicits an algorithm that searches a polynomial ``which
may be seen as the minimal polynomial for the ongoing computation''.

\begin{lemma}
\label{lemGD0}
Let $R\subseteq S$ be two entire rings with $S$ integral over $R$
and $R$ integrally closed. Let $I$ be a radical ideal of $R$
and $x \in IS$. There exists a monic polynomial $M(X)$
whose all non leading coefficients are in $I$ and such
that $M(x) =  0$. Let $P$ be another polynomial in $R[X]$
such that $P(x) =  0$. Let $K$ be the quotient field of $R$
and $Q$ the monic gcd of $P$ and $M$ in $K[X]$. Then $Q$
has its non leading coefficients in $I$, is such that
$Q(x) =  0$ and divides $M$ and $P$ in $R[X].$
\end{lemma}
\begin{proof}{Proof}
The existence of $M$ is easy. We write $x= \sum a_kb_k$ with $a_k$
in $I$ and $b_k$ in $B$, and consider the sub $R$-algebra
$T$ of $S$ generated by the $b_k$. This is a faithfull \tf $R$-module.
We write the matrix of the multiplication by $x$ over
a system of generators of the $R$-module $C$, which has all its
coefficients in $I$, and we take its characteristic polynomial.\\
Let $L$ be the quotient field of  $S$. Given the Bezout relation
 $UP+VM= Q$,
we have $Q(x)= 0$ in $K[x]\subseteq L$ and hence in $S$. Given the 
relation
$QS_1= S$, and since $R$ is integrally closed, $Q$ and $S_1$ have their
non leading coefficients in $\sqrt{I}= I$. Finally the quotient 
$P_1= P/Q$
can be computed in $R[X]$ by Euclidean division.
\end{proof}
NB: this lemma can be seen as the computational content of the
lemma in classical mathematics that the monic minimal polynomial
of $x$ has all its non leading coefficients in $I$.

\begin{proposition}
\label{propGD1}  {\em (Going Down in one step)}
Let $R\subseteq S$ be two integral rings with $S$ entire over $R$
and $R$ integrally closed.
Let $\cQ_1$ be a saturated \proi in  $S$, $\cP_1= (I_1,U_1)$ its trace 
on
$R$ and
$\cP_0= (I_0,U_0)$ a saturated \proi of $R$ with $I_0\subseteq I_1$.
If the \proc $(\cP_0,\cQ_1)$ collapses in $S$, then the
\proi $\cQ_1$ collapses in $S$ (and a fortiori the \proi  $\cP_1$ 
collapses
in $R$).
\end{proposition}
\begin{proof}{Proof}
Take $\cQ_1= (J_1,V_1)$.
Since $\cQ_1$ is complete, we can write the collapsus on the form
$$ u_0v_1= j_0\qquad {\rm avec}\; u_0\in U_0,\; v_1\in V_1,\;
j_0\in I_0B
$$
We know that $j_0$ cancels a monic polynomials $A$ with non leading
coefficients in $I_0$
$$ A(X)= X^k+\sum_{i<k}\; a_iX^i\qquad {\rm avec}\; a_i\in I_0
$$
hence $v_1$ cancels $M(u_0X)$
$$ 0= R(u_0v_1)= u_0^kv_1^k+\sum_{i<k}\; (a_iu^i)v_1^i
$$
We know also that $v_1$ cancels a monic polynomial $B$ with coefficients 
in
$R$ of degre $d$.\\
{\em First case:} $u_0^dS(X)= A(u_0X)$. Then the non leading 
coefficients of
$B$, the  $b_i= a_i/u_0^{k-i}$ are in $I_0$, since $\cP_0$ is saturated 
in
$R$. Hence $v_1^k\in I_0S$, so $v_1\in
\sqrt{I_0S}\subseteq\sqrt{I_1S}\subseteq J_1 $ ($I_0\subseteq 
I_1\subseteq
J_1$ and $J_1$ is a radical ideal of $S$). Hence $v_1\in V_1\cap J_1$.
\\
{\em Second case:} We don't have $u_0^dS(X)= A(u_0X)$. We shall
reduce this case to the first one. We apply the previous
lemma with $v_1$ and the ideal $R$. We get that $v_1$ cancels
the polynomial $B_1$ monic gcd of $A(u_0X)$ and $B(X)$, and
with coefficients in $R$, Let $d_1$ be the degre of $S_1$. We
consider the following monic polynomial with coefficients in $R$
$A_1(X)= u_0^{d_1}S_1(X/u_0)$. We have $A_1(j_0)= 0$,
$A_1(u_0X)= u_0^{d_1}B_1(X)$,  $A_1(X)$ is the monic gcd of $R(X)$ and
$u_0^dB(X/u_0)$.
Applying the previous lemma with  $j_0$ and the ideal $I_0$ we get that
$A_1$ has its non leading coefficients in $I_0$. We are thus back, with
$A_1$ and $B_1$, to the first case.
\end{proof}
Notice that the beginning of the proof (before the analysis of the 
second
case, that Sharp avoids by considering the minimal polynomial of $j_0$)
is almost the same as the proof in Sharp's book. Sharp does not use
the word ``collapsus'', and has for hypothesis that $\cQ_1$ is given
as a prime ideal, and shows that it would be absurd to have an equality
 $u_0v_1= j_0$ since this would imply
 $v_1\in V_1\cap J_1$ which is contradictory with his hypothesis.
This is a good case showing that our work consists essentially
to explicitate algorithms that are only implicit in classical arguments.
This illustrates also well a systematic feature of classical proofs, 
which
inverses positive and negative by the introduction of abstract objects.
The collapsus, which is a concrete fact, is seen as a negative fact
(``this would be absurd since we have a prime ideal'') while the
absence of collapsus, which requires a priori an infinity of 
verifications,
and hence is by essence negative, is felt as a positive fact, which
ensures the existence of abstract object. The price to pay for the
apparent comfort provided by these abstract objects is to transform
constructive arguments by non constructive one, changing the
(constructive) direct proof of
$P\Rightarrow Q$ by a proof by contradiction of
$\lnot Q\Rightarrow \lnot P$, \cad
a proof of $\lnot\lnot P\Rightarrow \lnot\lnot Q$.
\begin{theorem}
\label{thGD} {\em (Going Down)}
Let $R\subseteq S$ be two entire rings with $S$ integral over $R$
and $R$ integrally closed.
Let $\cC_1$ be a saturated \proc of $R$ and $\cC_2$
a saturated \proc of $S$, non empty. Let $I_\ell$ be the last of the 
ideals
in the \proc
$\cC_1$ and $I_{\ell+1}$ the first of the ideals in the \proc
$\cC_2|_R$. We assume $I_\ell\subseteq I_{\ell+1}$.
If the \proc $\cC_1\bullet \cC_2$ collapses in $S$, then the \proc 
$\cC_2$
collapses in $S$.
\end{theorem}
\begin{proof}{Proof}
If $\ell\ge 1$ and $r\ge1$ are the numbers of \prois in the \procs 
$\cC_1$
and $\cC_2$, write $GD_{\ell,r}$ the property that we want to establish.
We know already $GD_{1,1}$ from the Going Down in one step.\\
Since the  \procs  $\cC_2$ and  $\cC_2|_R$ are saturated, and by fact
\ref{factSatur}~(2) only the first \prois of the second \procz matters 
for
the collapsus.
Thus it is enough to show  $GD_{\ell,1}$, \cad the case where $\cC_2$
contains only one \proiz. We proceed by induction on $\ell$.
Take $\cC_1= (\cP_1,\ldots,\cP_\ell)$ a staurated \proc in  $R$,
($\cP_k= (I_k,U_k)$) and $\cQ$ a saturated
\proi in $S$.
Assume that $\cC_1\bullet \cQ$ collapses in  $S$.
Let $\cC= (\cP_2,\ldots,\cP_\ell,\cQ)$ and
$\cC'$ its saturation in $S$. If $\cQ_2= (J_2,V_2)$ 
is the first of the \prois in
$\cC'$ we have $I_1\subseteq I_2\subseteq (J_2\cap R)$, and we can apply
the Going Down in one step (or more precisely  $GD_{1,\ell}$):
$\cC'$ collapses in $S$. Hence $\cC$ collapses in $S$. One can
then apply the induction hypothesis.
\end{proof}

\begin{corollary}
\label{corGD1}
{\rm (Going Down, classical version)} \Tcgi
Let $R\subseteq S$ be two entire rings with $S$ integral over $R$
and $R$ integrally closed.
Let  $Q_{\ell+1}\subseteq\cdots\subseteq Q_{\ell+r}$ be prime ideals in
$S$,  $P_i= Q_i\cap A$  $(i= \ell+1,\ldots,\ell+ r)$, and
$P_{1}\subseteq\cdots\subseteq P_{\ell}$
be prime ideals in $R$ with $P_{\ell}\subseteq P_{\ell+1}$.
There exist  $Q_{1},\ldots,Q_{\ell}$,  prime ideals in $S$  which
satisfy $Q_1\subseteq\cdots\subseteq Q_{\ell}\subseteq Q_{\ell+1}$ and
$Q_{j}\cap A= P_{j}$ pour  $j= 1,\ldots,\ell$.
\end{corollary}
\begin{proof}{Proof}
Like in the proof of corollary \ref{corGU1}.
\end{proof}

We finish with a Going Down theorem for flat extensions.
(cf. \cite{Mat}).
\begin{theorem}
\label{thGDplat} {\em (Going Down for flat extensions)}
Let $R\subseteq S$ be two comutative rings with $S$ flat over $A$.
\begin{itemize}
\item [$(1)$] Let $\cQ_1= (J_1,V_1)$ be a saturated \proi of $S$,
$\cP_1= (I_1,U_1)$ its trace on $R$ and $\cP_0= (I_0,U_0)$ a saturated 
\proi of
$A$ with $I_0\subseteq I_1$.
It the \proc $(\cP_0,\cQ_1)$ collapses in $S$, then the
\proi $\cQ_1$ collapses in $S$ (and a fortiori the \proi  $\cP_1$ 
collapses
in $R$).
\item [$(2)$] Let $\cC_1$ be a saturated \proc of $R$ and $\cC_2$
a saturated \proc of $S$, non empty. Let $I_\ell$ be the last of the 
ideals
in the \proc
$\cC_1$ and $I_{\ell+1}$ the first of the ideals in the \proc
$\cC_2|_R$. Assume $I_\ell\subseteq I_{\ell+1}$.
If the \proc $\cC_1\bullet \cC_2$ collapses in $S$, then the \proc 
$\cC_2$
collapses in $S$.
\end{itemize}
\end{theorem}
\begin{proof}{Proof}
It is enough to show (1), because we can then show (2) like
in the proof of theorem \ref{thGD}.
Take $J_0= I_0B$. If $(\cP_0,\cQ_1)$ collapses in $S$, we have $v_1\in 
V_1$,
$u_0\in U_0$ and $j_0\in J_0$ with $v_1u_0+j_0= 0$. We write
$j_0= i_{1}b_1+\cdots+i_{r}b_r$ with the $i_{k}\in I_0$ and $b_{k}\in 
B$.
We get that the elements
$v_1,\; b_1,$ $\ldots,b_r$
of $S$ are linearly dependent over $R$.
$$ (u_0,i_{1},\ldots,i_{r})
\pmatrix{
    v_1     \cr
    b_1     \cr
   \vdots      \cr
    b_r
} = 0
$$
Since $S$ is flat over $R$ this relation can be decomposed to
$$ (u_0,i_{1},\ldots,i_{r})M= (0,\ldots,0)\; \; {\rm and} \; \;
\pmatrix{
    v_1     \cr
    b_1     \cr
   \vdots      \cr
    b_r
} =  M
\pmatrix{
    b'_1     \cr
    b'_2     \cr
   \vdots      \cr
    b'_s
}
$$
where $M= (m_{k,\ell})\in R^{s{\times}(r+1)}$ and the $b'_k$ are in 
$S$.
Each relation
$$ u_0m_{0,\ell}+i_{1}m_{1,\ell}+\cdots+i_{r}m_{1,\ell}= 0
$$
implies that $m_{0,\ell}\in I_0$ since $\cP_0$ is saturated.
A fortiori $m_{0,\ell}\in J_1$. Hence the relation
$$ v_1= m_{0,1}b'_1+\cdots+m_{0,s}b'_s
$$
is a collapsus $\cQ_1$ of $S$.
\end{proof}

For this proof, that was literally forced upon us, we did not try to
analyse the quite abstract argument given by Matsumara for
the following corollary.

\begin{corollary}
\label{corGD1plat}
{\rm (Going Down for flat extension, classical version)} \Tcgi
Let $R\subseteq S$ be two commutative rings with $S$ flat over $R$.
Let $Q_{\ell+1}\subseteq\cdots\subseteq Q_{\ell+r}$ be prime ideals
in $S$,  $P_i= Q_i\cap R$  $(i= \ell+1,\ldots,\ell+ r)$, and
$P_{1}\subseteq\cdots\subseteq P_{\ell}$
be prime ideals in $R$ with $P_{\ell}\subseteq P_{\ell+1}$.
There exist $Q_{1},\ldots,Q_{\ell}$, prime ideals in $S$
such that  $Q_1\subseteq\cdots\subseteq Q_{\ell}\subseteq Q_{\ell+1}$ 
and
$Q_{j}\cap R= P_{j}$ for  $j= 1,\ldots,\ell$.
\end{corollary}


\newpage
\setcounter{section}{1}\setcounter{subsection}{0}
\setcounter{theorem}{0}
\def\thesection{\Alph{section}}
\addcontentsline{toc}{section}{Annex: The \tcg and LLPO}
\section*{Annex: Completeness, compactness theorem, LLPO and geometric theories}

\subsection{Theories and models}

  We fix a set $V$ of {\em atomic propositions} or
{\em propositional letters}. A proposition $\phi,\psi,\dots$
is a syntactical
object built from the atoms $p,q,r\in V$ with the
usual logical connectives
$$ 0,\;\;1,\;\;\phi\wedge \psi,\;\;\phi\vee\psi,\;\;\phi\rightarrow\psi,\;\;
\neg\phi$$
We let $P_V$ be the set of all propositions. Let $F_2$
be the Boolean algebra with two elements.
A {\em valuation}
is a function $v\in F_2^V$ that assigns a truth value
to any of the atomic propositions. Such a valuation can
be extended to a map
$P_V\rightarrow \{0,1\},\;\phi\longmapsto v(\phi)$
in the expected way. A {\em theory} $T$ is a subset of $P_V$.
A {\em model} of $T$ is a valuation $v$ such that $v(\phi)=1$
for all $\phi\in T$.

 More generally given a Boolean algebra $B$ we can define
$B$-valuation to be a function $v\in B^V$. This can be extended
as well to a map $P_V\rightarrow B,\;\phi\longmapsto v(\phi)$.
A {\em $B$-model} of $T$ is a valuation $v$ such that $v(\phi)=1$
for all $\phi\in T$. The usual notion of model is a direct special
case, taking for $B$ the Boolean algebra $F_2$.
For any theory there exists always a
free Boolean algebra over which $T$ is a model, the {\em Lindenbaum}
algebra of $T$, which can be also be defined as the Boolean
algebra generated by $T$, thinking of the elements of $V$ as
generators and the elements of $T$ as relations. The theory $T$
is {\em formally consistent} if, and only if, its Lindenbaum algebra is not trivial.

\subsection{Completeness theorem}

\begin{theorem}
(Completeness theorem) Let $T$ be a theory. If $T$ is formally
consistent then $T$ has a model.
\end{theorem}

 This theorem is the completeness theorem for propositional logic.
Such a theorem is strongly related to Hilbert's program, which
can be seen as an attempt to replace the question of existence of
model of a theory by the formal fact that this theory
is not contradictory.

Let $B$ the Lindenbaum algebra of $T$. To prove completeness, it is
enough to find a morphism $B\rightarrow F_2$ assuming that $B$ is not
trivial, wich is the same as finding a prime ideal (which is then
automatically maximal) in $B$. Thus the completeness theorem is a
consequence of the existence of prime ideal in nontrivial Boolean algebra.  Notice
that this existence is clear in the case where $B$ is finite, hence
that the completeness theorem is direct for finite theories.

\subsection{Compactness theorem}

 The completeness theorem for an arbitrary theory
can be seen as a corollary of the following
fundamental result.

\begin{theorem}
(Compactness theorem) Let $T$ be a theory.
If all finite subsets of $T$ have a model then so does $T$.
\end{theorem}

 Suppose indeed that the compactness theorem holds, and let
$T$ be a formally consistent theory. Then an arbitrary finite
subset $T_0$ of $T$ is also formally consistent. Furthermore, we
have seen that this implies the existence of a model for $T_0$.
 It follows then from the compactness theorem
that $T$ itself has a model.

 Conversely, it is clear that the compactness theorem follows
from the completeness theorem, since a theory is formally
consistent as soon as all its finite subsets are.

 A simple general proof of the compactness theorem
is to consider the product topology
on $\{0,1\}^V$ and to notice that the set of models
of a given subset of $T$ is a closed subset. The theorem is then
a corollary of the compactness of the space $W:=\{0,1\}^V$ when
compactness is expressed (in classical mathematics) as:
if a family of closed subsets of $W$ has non-void finite
intersections, then its intersection is non-void.

\subsection{LPO and LLPO}

If $V$ is countable (\cad discrete and enumerable)
we have the following
alternative argument. One writes $V = \{p_0,p_1,\dots \}$
and builds by induction a partial valuation $v_n$
on $\{p_i~|~i<n\}$ such that any finite subset of $T$ has
a model which extends $v_n$, and $v_{n+1}$ extends
$v_n$. To define $v_{n+1}$ one first tries $v_{n+1}(p_n) = 0$.
If this does not work, there is a finite subset of $T$
such that any of its model $v$ that extends $v_n$ satisfies
$v(p_n) = 1$ and one can take $v_{n+1}(p_n) = 1.$

The non-effective part of this argument is contained in
the choice of $v_{n+1}(p_n)$, which demands to give a
gobal answer to an infinite set of (elementary) questions.

Now let us assume also that we can enumerate the infinite
set $T$. We can then build a sequence of finite
subsets of $T$ in a nondecreasing way
$K_0\subseteq K_1\subseteq \dots$ such that any finite
subset of $T$ is a subset of some $K_n$.
Assuming we have construct $v_{n}$ such that all $K_j$'s
have a model extending $v_{n}$, in order to define
$v_{n+1}(p_n)$ we have to give a global answer to the
questions: do all $K_j$'s have a model extending $v_{n+1}$
when we choose $v_{n+1}(p_n)=1$~?
For each $j$ this is an elementary question, having
a clear answer.
More precisely let us define $g_{n}:\N\rightarrow \{0,1\}$
in the following way: $g_{n}(j)=1$ if there is a model
$v_{n,j}$ of $K_{j}$ extending  $v_{n}$ with $v_{n,j}(p_{n})=1$,
else $g_{n}(j)=0$. By induction hypothesis if $g_{n}(j)=0$
then all  $K_{\ell}$ have a model $v_{n,\ell}$ extending
$v_{n}$ with $v_{n,\ell}(p_{n})=1$, and all models  $v_{n,\ell}$
of  $K_{\ell}$ extending  $v_{n}$ satisfy  $v_{n,\ell}(p_{n})=1$
if $\ell\geq j$.
So we can ``construct" inductively
the infinite sequence of partial models $v_{n}$ by using
at each step the non-constructive Bishop's principle LPO
(Least Principle of Omniscience):
given a function $f:\N\rightarrow \{0,1\}$,
either $f=1$ or $\exists j\in\N\;f(j)\neq 1$.
This principle is applied at step $n$ to the function $g_n$.

In fact  we can slightly modify the argument and
use only a combination of Dependant Choice and
of Bishop's principle LLPO (Lesser Limited
Principle of Omniscience), which is known to be
strictly weaker than LPO: given two non-increasing functions
$g,h:\N\rightarrow \{0,1\}$ such that, for all $j$
$$g(j) = 1 \vee h(j) = 1$$
then we have $g=1$ or $h=1$.
Indeed let us define  $h_{n}:\N\rightarrow \{0,1\}$
in a symmetric way: $h_{n}(j)=1$ if there is a model
$v_{n,j}$ of $K_{j}$ extending  $v_{n}$ with $v_{n,j}(p_{n})=0$,
else $h_{n}(j)=0$. Cleraly $g_n$ and $h_n$ are non-increasing functions.
By induction hypothesis, we have for all $j$
$g_n(j) = 1 \vee h_n(j) = 1$. So,
applying LLPO, we can define $v_{n+1}(p_n)=1$
if $g_n=1$ and $v_{n+1}(p_n)=1$ if $h_n=1$.
Nevertheless, we have to use dependant choice in
order to make this choice inifnitely often since the answer
``$g=1$ or $h=1$" given by the oracle LLPO may be
ambiguous.

In a reverse way
it is easy to see that the \tcg restricted to the
countable case implies LLPO.

\subsection{Geometric formulae and theories}

{\em What would have happened if topologies without points
had been discovered before topologies with points, or if
Grothendieck had known the theory of distributive lattices?}  (G. C. Rota \cite{Rota}).


\bigskip

 A formula is {\em geometric} if, and only if, it is built only
with the connectives $0,1,\phi\wedge\psi,\phi\vee\psi$ from
the propositional letters in $V$. A theory if
a (propositional) {\em geometric} theory iff all the formula in $T$
are of the form $\phi\rightarrow \psi$ where $\phi$
and $\psi$ are geometric formulae.

 It is clear that the formulae of a geometric theory $T$ can be seen as
relations for generating a distributive lattice $L_T$ and that
the Lindenbaum algebra of $T$ is nothing else but the free Boolean
algebra generated by the lattice $L_T$. It follows from
Proposition \ref{propTrBoo} that $T$ is formally consistent if, and
only if, $L_T$ is nontrivial. Also, a model of $T$ is nothing else
but an element of $\Spec(L_T)$.

\begin{theorem}
(Completeness theorem for geometric theories) Let $T$ be a geometric
theory. If $T$ generates a nontrivial distributive lattice, then
$T$ has a model.
\end{theorem}

 The general notion of geometric formula allows also existential
quantification, but we restrict ourselves here to the propositional case.
Even in this restricted form, the notion of geometric theory is
fundamental. For instance, if $R$ is a commutative ring, we can
consider the theory with atomic propositions
$D(x)$ for each $x\in R$ and with axioms
\begin{itemize}
\item $D(0_R)\rightarrow 0$
\item $1\rightarrow D(1_R)$
\item $D(x)\wedge D(y)\rightarrow D(xy)$
\item $D(xy)\rightarrow D(x)$
\item $D(x+y)\rightarrow D(x)\vee D(y)$
\end{itemize}
This is a geometric  theory $T$. The model of this theory are
clearly the complement of the prime ideals. What is remarkable is that, while the
existence of models of this theory is a nontrivial fact
which may be dependent on set theoretic axioms (such as dependent axiom of choices)
its formal consistency is completely elementary (as explained in the
beginning of the section \ref{secZariKrull}).
This geometric theory, or the distributive lattice it generates,
can be seen as a point-free description of the Zariski
spectrum of the ring. The distributive lattice
generated by this theory (called in this paper the Zariski lattice of $R$)
is isomorphic to the lattice
of compact open of the Zariski spectrum of $R$, while the Boolean
algebra generated by this theory is isomorphic to the algebra
of the constructible sets.



\newpage
\begin{thebibliography}{50}
\addcontentsline{toc}{section}{Bibliographie}
\bibitem{clr} Coste M., Lombardi H., Roy M.-F. {\it  Dynamical method in
algebra: Effective Nullstellens\"atze} Annals of Pure and Applied Logic
{\bf 111}, (2001) 203--256

\bibitem{cc}  Cederquist, Coquand T. {\em Entailment relations and
Distributive Lattices} Logic Colloquium '98 (Prague), 127--139,
Lect. Notes Log., 13. Assoc. Symbol. Logic, Urbana, (2000).


\bibitem{cp}  Coquand T., Persson H. {\it  Valuations and Dedekind's
Prague Theorem}.
J. Pure Appl. Algebra 155 (2001), no. 2-3, 121--129.

\bibitem{esp} Espan\~ol L. {\em Constructive Krull dimension of
lattices.}
Rev. Acad. Cienc. Zaragoza (2) 37 (1982), 5--9.

\bibitem{joy} Joyal A. {\em Le th\'eor\`eme de Chevalley-Tarski.}
Cahiers de Topologie et G\'eometrie Differentielle, 1975.

\bibitem{kl} Kuhlmann F.-V., Lombardi H.
{\it Construction du hens\'elis\'e d'un corps valu\'e.}
Journal of Algebra  {\bf 228}, (2000),
624--632.

\bibitem{lom95} Lombardi H.
{\it Un nouveau positiv\-stellen\-satz effectif pour les corps
valu\'es.}
S\'eminaire ``Structures Ordonn\'ees" (Paris 6-7)
1996. Es: F. Delon, A. Dickmann, D. Gondard)

\bibitem{lom97} Lombardi H.   {\it Le contenu constructif d'un principe
local-global avec une
application \`a la structure d'un module projectif de type fini }.
Publications Math\'ematiques
de Besan\c{c}on. Th\'eorie des nombres. (1997).
Fascicule  94--95 \& 95--96.

\bibitem{lom98} Lombardi H. {\it  Relecture constructive de la th\'eorie
d'Artin-Schreier}. Annals of Pure and Applied Logic,  {\bf 91}, (1998),
59--92.

\bibitem{lom}  Lombardi H. {\em Dimension de Krull, Nullstellens\"atze
 et \'Evaluation dynamique}.  To appear in Math. Zeitschrift.

\bibitem{lom99} Lombardi H. {\it Platitude, localisation et anneaux de
Pr\"ufer: une approche constructive}. 
Publications Math\'ematiques de Besan\c{c}on.
Th\'eorie des nombres. Ann\'ees 1998-2001.

\bibitem{lom99a}   Lombardi H.
{\em Constructions cach\'ees en alg\`ebre abstraite (1)
Relations de d\'ependance int\'egrale}.
To appear in Journal of Pure and Applied Algebra.

\bibitem{lq99}  Lombardi H., Quitt\'e C.
{\em Constructions cach\'ees en alg\`ebre abstraite (2)
Th\'eor\`eme de Horrocks, du local au global}.  
to appear in 
Commutative ring theory and applications. Eds: Fontana M., 
Kabbaj S.-E., Wiegand S.

\bibitem{Macneille1} MacNeille H. M. {\em Partially ordered sets.}
Trans. Amer. Math. Soc. 42 (1937), no. 3, 416--460.

\bibitem{Mat} Matsumura H.: {\it Commutative ring theory}.
Cambridge studies in advanced mathematics \num8.
Cambridge University Press. 1989

\bibitem{MRR} Mines R., Richman F., Ruitenburg W. {\it A Course in
Constructive Algebra.} Universitext. Springer-Verlag, 1988.


\bibitem{Rota} Rota G.C. {\it Indiscrete Thoughts} Birkhauser, 1995.

\bibitem{Sha} Sharp  {\it Steps in Commutative Algebra.}
L.M.S. Student Texts 19. Cambridge University Press.

\end{thebibliography}
\end{document}